\documentclass[12pt,draft,leqno]{article}
\usepackage{amssymb,eucal,latexsym}
\usepackage{amsmath,amsthm,latexsym,amscd,amsbsy,fleqn,graphicx}
\usepackage{xy}
\xyoption{all}

\newtheorem{theorem}{Theorem}[section]
\newtheorem{lm}[theorem]{Lemma}
\newtheorem{exa}[theorem]{Example}

\newtheorem{cor}[theorem]{Corollary}
\newtheorem{pro}[theorem]{Proposition}
\newtheorem{defi}[theorem]{Definition}

\newtheorem{nota}[theorem]{Notation}

\newtheorem{rem}[theorem]{Remark}

\newtheorem{fact}[theorem]{Fact}

\newtheorem{nist}[theorem]{}

\newcommand{\tiff}{if and only if \ }
\newcommand{\df}{\ensuremath{\overset{\mathrm{df}}{=}}}
\newcommand{\klam}[1]{\ensuremath{\langle #1 \rangle}}

\def\p{\varphi}
\def\a{\alpha}
\def\b{\beta}
\def\d{\delta}
\def\ep{\varepsilon}

\def\DE{\Delta}

\def\vk{\varkappa}

\def\l{\lambda}

\def\s{\sigma}

\def\hs{\hat{h}}

\def\lra{\longrightarrow}

\def\sbe{\subseteq}

\def\stm{\setminus}
\def\ems{\emptyset}
\def\nes{\neq\emptyset}

\def\we{\wedge}

\def\ap{^{\prime}}
\def\inv{^{-1}}
\def\st{\ |\ }

\def\card #1{\vert #1 \vert}

\def\Di{{\bf D}}

\def\1{{\bf 1}}
\def\2{\mbox{{\bf 2}}}
\def\3{\mbox{{\bf 3}}}

\def\bi{\breve{\imath}}

\def\AA{{\cal A}}
\def\BB{{\cal B}}
\def\CC{{\cal C}}

\def\KK{{\cal K}}

\def\OO{{\cal O}}

\def\TT{{\cal T}}

\def\AAA{{\sf A}}

\def\FFF{{\sf F}}
\def\GGG{{\sf G}}

\def\PPP{{\sf P}}

\def\SSS{{\sf S}}
\def\TTT{{\sf T}}

\def\co{{\sf CO}}
\def\Id{{\sf Id}}

\def\ZComp{{\bf ZComp}}
\def\EDComp{{\bf EDComp}}
\def\Comp{{\bf Comp}}
\def\Tych{{\bf Tych}}
\def\MDeVe{{\bf MDeVe}}
\def\DeVe{{\bf DeVe}}

\def\ZHLC{{\bf BooleSp}}

\def\PZHLC{{\bf BooleSp}}
\def\GBPL{{\bf GBoole}}
\def\ZHC{{\bf Stone}}
\def\Bool{{\bf Boole}}

\def\MBool{{\bf mzMaps}}
\def\ZBool{{\bf zMaps}}
\def\CMZM{{\bf cmzMaps}}
\def\CZM{{\bf czMaps}}
\def\TMBool{{\bf TMaps}}
\def\KMBool{{\bf kMaps}}

\def\ZA{{\bf zBoole}}
\def\DZA{{\bf dzBoole}}
\def\CDZA{{\bf kBoole}}
\def\TDZA{{\bf TBoole}}

\def\dom{{\rm dom}}

\def\DZCB{{\bf dzCBoole}}
\def\ZCB{{\bf zCBoole}}
\def\ZB{{\bf zBoole}}
\def\EDT{{\bf EDTych}}

\def\int{\mbox{{\rm int}}}

\def\cl{\mbox{{\rm cl}}}

\def\CO{\mbox{{\rm CO}}}
\def\RC{\mbox{{\rm RC}}}

\def\At{{\rm At}}
\def\Att{{\sf At}}

\def\Tych{{\bf Tych}}
\def\Top{{\bf Top}}

\def\ZH{{\bf ZHaus}}

\def\Caba{{\bf Caba}}

\def\Set{{\bf Set}}

\def\tcx{t_X^C}
\def\tcy{t_Y^C}
\def\tcx0{t_{(X,X_0)}}
\def\tcy0{t_{(Y,Y_0)}}

\def\CLO{{\rm CO}}
\def\RCO{{\rm RC}}

\def\bU0{\bar{U}=(U^0,(U^i,U^{ci})_{i\in\omega})}

\def\bV0{\bar{V}=(V^0,(V^i,V^{ci})_{i\in\omega})}

\title{{\LARGE\bf Two extensions of the Stone Duality}\\
\vspace{0.2cm} {\LARGE\bf to the category of zero-dimensional}\\
\vspace{0.2cm} {\LARGE\bf Hausdorff spaces}\\ \vspace{0.5cm}
{\large\bf Georgi Dimov and Elza Ivanova-Dimova}\thanks{The authors
 were supported by the Bulgarian National Fund of Science, contract no. DN02/15/19.12.2016.}
\\
\vspace{0.2cm} {\footnotesize\rm Faculty of Math. and Informatics,
Sofia University,} {\footnotesize\rm 5 J. Bourchier Blvd., 1164
Sofia, Bulgaria}
}

\author{}

\date{}

\begin{document}

\maketitle

\begin{abstract}
Extending the Stone Duality Theorem, we prove two duality theorems for
the category $\ZH$ of
zero-dimensional Hausdorff spaces and continuous maps. Both of them imply easily the
Tarski Duality Theorem,
as well as two new duality theorems for the category $\EDT$ of extremally disconnected Tychonoff spaces and continuous maps.
Also, we describe  two categories which are dually equivalent to the category
$\ZComp$ of zero-dimensional Hausdorff compactifications of
zero-dimensional Hausdorff spaces and obtain as a corollary the Dwinger Theorem about zero-dimensional compactifications of a zero-dimensional
Hausdorff space.
\end{abstract}



\footnotetext[3]{{\footnotesize {\em E-mail addresses:}
gdimov@fmi.uni-sofia.bg, elza@fmi.uni-sofia.bg}}

\section{Introduction}

In 1937, M. Stone \cite{ST} proved that there exists a bijective
correspondence $T_l$ between the class of all (up to
homeomorphism) zero-dimensional locally compact Hausdorff spaces
(briefly, {\em Boolean spaces}\/) and the class of all (up to
isomorphism) generalized Boolean algebras (briefly, GBAs) (or,
equivalently, Boolean rings with or without unit). In the class of
compact Boolean spaces (briefly, {\em Stone spaces}\/) this
bijection can be extended to a dual equivalence $\TTT :\ZHC\lra\Bool$ between
the category $\ZHC$ of Stone spaces and continuous maps and the
category $\Bool$ of Boolean algebras and Boolean homomorphisms;
this is the classical Stone Duality. In 1964, H. P. Doctor
\cite{Do} showed that the Stone bijection $T_l$ can be even
extended  to a dual equivalence between the category $\PZHLC_{\bf perf}$ of
 Boolean spaces and  perfect maps between them and the
category $\GBPL$ of  GBAs and suitable morphisms between them.
Later on, G. Dimov \cite{D-a0903-2593,D-PMD12} extended the Stone
Duality to the category $\ZHLC$ of
Boolean spaces and continuous maps.

In this article, which was inspired by the recent paper \cite{BMO}
of G. Bezhanishvili,  P. J. Morandi  and  B. Olberding,  we describe two extensions of
the Stone Duality to the category $\ZH$ of zero-dimensional
Hausdorff spaces and continuous maps. Namely, we  define two
categories   $\DZA$ and $\MBool$, and prove
that they are dually equivalent to the category $\ZH$. From the restrictions of our dual equivalences $F:\ZH\lra\DZA$ and $\FFF:\ZH\lra\MBool$
to the category $\ZHC$  and the category $\Di$ of discrete spaces and continuous maps, we obtain easily the Stone Duality and the Tarski Duality, respectively. The restrictions of $F$ and $\FFF$ to the category $\EDT$ of extremally disconnected Tychonoff spaces and continuous maps give us two duality theorems for the category $\EDT$. We
introduce as well two other categories, namely, the categories
$\ZA$ and $\ZBool$, and show that they are dually equivalent to
the category $\ZComp$ of zero-dimensional Hausdorff
compactifications of zero-dimensional Hausdorff spaces.
As a corollary, we obtain the Dwinger Theorem \cite{Dwinger} about zero-dimensional compactifications of a zero-dimensional
Hausdorff space. Let us note that the category  $\ZComp$ is a full subcategory of the
category $\Comp$ of all Hausdorff compactifications of Tychonoff
spaces defined in \cite{BMO}.

The paper is organized as follows.  Section 2 contains all
preliminary facts and definitions which are used in this paper.

In
Section 3, we introduce the notions of {\em Boolean z-algebra} and
{\em Boolean dz-algebra}, define the category $\DZA$ having as objects  all dz-algebras and
prove our first duality theorem for the category $\ZH$ by showing that there exist contravariant functors
$F:\ZH\lra\DZA$  and $G:\DZA\lra\ZH$  whose compositions are naturally isomorphic to the corresponding identity functors
(see Theorem \ref{zduality}).

In the next Section 4, we introduce the notions of {\em Boolean z-map} and {\em maximal
Boolean z-map}, and define the category
 $\MBool$ having as objects  all
maximal Boolean z-maps.  In Theorem \ref{zeq} we show that the categories
$\DZA$ and $\MBool$ are equivalent. This implies immediately that
the categories $\ZH$ and $\MBool$ are dually equivalent (see
Theorem  \ref{nzduality} which is our second duality theorem for the category $\ZH$). The corresponding dual equivalences are denoted by
$\FFF:\ZH\lra\MBool$ and $\GGG:\MBool\lra\ZH$.

In Section 5 we describe the subcategories of the categories $\DZA$ and $\MBool$ which are isomorphic to the category $\Bool$ (see Propositions \ref{cdza} and \ref{mzbool}) and show that the corresponding restrictions of $F$, $G$, $\FFF$ and $\GGG$ imply the Stone Duality Theorem (see Propositions \ref{extstd} and \ref{extstdm}). Thus $F$, $G$, $\FFF$ and $\GGG$ are extensions of the classical Stone dual equivalences  $\TTT:\ZHC\lra\Bool$ and $\SSS:\Bool\lra\ZHC$.

In Section 6 we describe the subcategories of the categories $\DZA$ and $\MBool$ which are isomorphic to the category $\Di$ (see Proposition \ref{resttar}), prove that the corresponding restrictions of $F$, $G$, $\FFF$ and $\GGG$ lead to a  dual equivalence $$\AAA:\Caba\lra\Set$$ which is slightly different from the classical Tarski dual equivalence $\Att:\Caba\lra \Set$, and show that it implies the Tarski Duality Theorem (see Propositions \ref{newtar}).

In Section 7 we regard the restrictions of $F$, $G$, $\FFF$ and $\GGG$ to the category $\EDT$ and obtain two duality theorems for the category $\EDT$ (see Theorems \ref{zdualityed} and \ref{mzdualityed}). The categories which a dually equivalent to the category $\EDT$ are simpler than the categories $\DZA$ and $\MBool$; their objects are all complete Boolean z-algebras and all complete Boolean z-maps, respectively, although one could expect that their objects should be all complete Boolean dz-algebras and all complete Boolean mz-maps, respectively.

In the last Section 8, we define the categories $\ZA$ and
$\ZBool$. Their objects are, respectively, all Boolean z-algebras
 and all Boolean z-maps. We show that the category $\ZA$ is dually equivalent to the category $\ZComp$
(see Theorem \ref{zdualityc}). Then we prove that  the categories $\ZComp$ and
$\ZBool$ are dually equivalent (see Theorem \ref{nzdualityc}).
In \ref{dwthcor} we show that both of these results imply the Dwinger Theorem \cite{Dwinger} which describes the ordered set of all, up to equivalence, zero-dimensional compactifications of a zero-dimensional Hausdorff  space $X$.

We want to add that in the continuation \cite{DD} of this paper, we show how the Dimov Duality Theorem for Boolean spaces \cite{D-a0903-2593,D-PMD12} can be derived from  our duality Theorem \ref{zduality} and, moreover, using our Theorems \ref{zduality} and \ref{nzduality}, we prove two new duality theorems for the category $\ZHLC$.

 We now fix the notation.

 Throughout, $(B, \land, \lor, {}^*, 0, 1)$ will denote a Boolean algebra unless indicated otherwise;
we do not assume that $0 \neq 1$. With some abuse of language, we
shall usually identify algebras with their universe, if no
confusion can arise.

 We denote by $\2$ the simplest Boolean
algebra containing only $0$ and $1$, where $0\neq 1$.

 If $A$ is a Boolean algebra, then  $A^+ \df A \setminus \{0\}$ and $\At(A)$ is the set of all atoms of $A$.

 If $X$ is a set, we denote by $P(X)$ the power set of $X$;
 clearly, $(P(X),\cup,\cap,\stm,\ems, X)$  $(=(P(X),\sbe))$ is a complete atomic Boolean
 algebra.

 If $X$ is a topological space, we denote by $\CO(X)$ the set of
all clopen (= closed and open) subsets of $X$. Obviously,
$(\CO(X),\cup,\cap,\stm,\ems, X)$ $(=(\CO(X),\sbe))$ is a Boolean algebra.

 If $(X,\TT)$ is a topological space and $M$ is a subset of $X$, we
denote by $\cl_{(X,\TT)}(M)$ (or simply by $\cl(M)$ or $\cl_X(M)$)
the closure of $M$ in $(X,\TT)$ and by $\int_{(X,\TT)}(M)$ (or
briefly by $\int(M)$ or $\int_X(M)$) the interior of $M$ in
$(X,\TT)$.

If $\CC$ is a category, we denote by $\card\CC$ the class of the objects of $\CC$ and by $\CC(X,Y)$ the set of all
   $\CC$-morphisms between two $\CC$-objects $X$ and $Y$.

We denote by:

\begin{itemize}

\item $\Set$ the category of  sets and  functions,

\item  $\Top$ the category of  topological spaces and continuous maps,

\item  $\ZH$ the category of all zero-dimensional Hausdorff spaces and continuous
maps,

\item  $\Di$ the category of all discrete spaces and continuous
maps,

\item $\ZHC$ the category of all compact Hausdorff zero-dimensional spaces (= {\em Stone spaces}) and their continuous maps,

\item $\EDT$ the category of extremally disconnected Tychonoff spaces  and continuous maps,

\item $\Bool$ the category
of  Boolean algebras and Boolean  homomorphisms,

\item $\Caba$ the category of all complete atomic Boolean algebras and all complete Boolean
homomorphisms between them.
\end{itemize}

   The main reference books for all notions which are not defined here
 are \cite{AHS,MacLane,Dwinger,E}.

 \section{Preliminaries}

\begin{nist}\label{hats}
\rm
Let $\a\in\Bool(A,B)$ and $x\in\At(B)$. Then it is easy to see that the map $$\a_x:A\lra \2$$ defined by
$\a_x(a)=1\Leftrightarrow x\le\a(a)$, for $a\in A$, is a Boolean homomorphism. We put $$X_\a\df\{\a_x\st x\in\At(B)\}.$$
{\em Note that if  $\a$ is a complete Boolean homomorphism, then, for every $x\in \At(B)$, $\a_x$ is a complete Boolean homomorphism as well.}
We put $$h_\a:\At(B)\lra X_\a,\ \ x\mapsto \a_x.$$ It is easy to see that {\em if  every atom of $B$ is a meet of some
elements of $\a(A)$, then $h_\a$ is a bijection.}

If $A=B$ and $\a=id_B$, then we have that
$\a_x(b)=1\Leftrightarrow x\le b$, for all $b\in B$. In this case, for simplicity, we will write $\check{x}$ instead of  $\a_x$,  $\check{X}_B$ instead of   $X_\a$ and $\check{h}_B$ instead of $h_\a$. Hence,
$$\check{x}:B\lra \2$$ is defined by $\check{x}(b)=1\Leftrightarrow x\le b$, for all $b\in B$,  $$\check{X}_B\df\{\check{x}\st x\in\At(B)\}$$
and
$$\check{h}_B:\At(B)\lra \check{X}_B,\ \ x\mapsto \check{x}.$$
{\em Note that every $\check{x}$ is a complete Boolean homomorphism and $\check{h}_B$ is a bijection.}

Further, if $X$ is a set, $B=\PPP(X)$, $A$ is a Boolean subalgebra of $B$ and $\a$ is the inclusion map, then, obviously, the map $\a_x$ is defined by $\a_x(U)=1\Leftrightarrow x\in U$, for every $U\in A$. In order to simplify the notation, for such $A$ and $B$, we will write $\hat{x}$ (and, sometimes, even $\hat{x}_A$) instead of $\a_x$.
({\em Note that every $\hat{x}$ is a complete Boolean homomorphism.}) Thus, in such a case, by $$\hat{x}:A\lra\2$$ we will understand the map defined by $\hat{x}(U)=1\Leftrightarrow x\in U$, for every $U\in A$; also, we will write $\hat{X}_A$  instead of $X_\a$, and $\hat{h}_{X,A}$ instead of $h_\a$, i.e.,  $$\hat{X}_A=\{\hat{x}:A\lra\2\st x\in X\}$$
and
$$\hat{h}_{X,A}:X\lra \hat{X}_A,\ \ x\mapsto \hat{x}.$$ Note that  if the family $A$ {\em $T_0$-separates the points of $X$} (i.e., for every $x,y\in X$ such that $x\neq y$,  there exists $U\in A$ with $|U\cap\{x,y\}|=1$), then {\em the map $\hat{h}_{X,A}$ is a bijection.}

 If $X$ is a topological space and $A=(\CO(X),\sbe)$, {\em we will simply write $\hat{X}$ instead of $\hat{X}_A$, and  $\hat{h}_X$ instead of $\hat{h}_{X,A}$, i.e.,} $$\hat{h}_{X}:X\lra \hat{X},\ \ x\mapsto \hat{x}.$$
 Obviously, {\em if $X$ is a zero-dimensional Hausdorff space, then $\hat{h}_{X}$ is a bijection.}
\end{nist}

\begin{nist}\label{nistone}
\rm We will denote by $\co:\Top\lra\Bool$ the contravariant functor which assigns to every $X\in|\Top|$ the Boolean algebra $(\CO(X),\sbe)$ and to every $f\in\Top(X,Y)$, the Boolean homomorphism $\co(f):\co(Y)\lra\co(X)$ defined by $\co(f)(U)\df f\inv(U)$, for every $U\in\CO(Y)$.

Now we will briefly describe  the Stone duality \cite{ST}
between the categories $\Bool$ and $\ZHC$ using its presentation
given in \cite{H}. We will define two contravariant functors
$$\SSS :\Bool\lra\ZHC \ \ \mbox{ and } \ \ \TTT :\ZHC\lra \Bool.$$ For any Boolean algebra $A$, we let the
space $\SSS (A)$  to be the set
$$X_A\df\Bool(A,\2)$$ endowed with a topology $\TT_A$ having as a
closed base the family $\{s_A(a)\st a\in A\}$, where
\begin{equation}\label{sofa}
s_A(a)\df\{x\in X_A\st x(a)=1\},
\end{equation}
for every $a\in A$; then $\SSS (A)= (X_A,\TT_A)$ is a
 Stone
space. Note that the family $\{s_A(a)\st a\in A\}$ is also an open base of
the space $(X_A,\TT_A)$.

If $\p\in\Bool(A,B)$, then we define $\SSS (\p): \SSS (B)\lra \SSS (A)$
by the formula $\SSS (\p)(y)\df y\circ\p$ for every $y\in \SSS (B)$.
It is easy to see that $\SSS $ is a contravariant functor.

The contravariant functor $\TTT$ is defined to be the restriction of the contravariant functor $\co$ to the category $\ZHC$.

For every $X\in|\ZHC|$, the map $t_X:X\lra \SSS (\TTT(X)), \ \
x\mapsto (\hat{x}:\co(X)\lra\2)$ is a homeomorphism and $$t:\Id_{\ZHC}\lra \SSS\circ\TTT, \ \  X\mapsto t_X,$$ is a natural isomorphism.
Also, {\em the Stone map}
\begin{equation}\label{stonemap}
s_A:A\lra \TTT (\SSS (A)), \ \ a\mapsto s_A(a),
\end{equation}
is a $\Bool$-isomorphism and $$s:\Id_{\Bool}\lra \TTT\circ \SSS, \ \ A\mapsto s_A,$$
is natural isomorphism.
Thus  $\klam{\TTT,\SSS,t,s}:\ZHC\lra\Bool$ is an adjoint  dual
equivalence (in the sense of \cite{MacLane}).
\end{nist}

\begin{defi}\label{niDE}
\rm An {\em extension} of a space $X$ is a pair $(Y,c)$, where $Y$
is a space and $c:X\lra Y$ is a dense embedding of $X$ into  $Y$. Often we will simply write $c$ instead of $(Y,c)$.

Two extensions $(Y_i,c_i),\ i=1,2$, of $X$ are called {\em equivalent} if there exists a homeomorphism
$f:Y_1\lra Y_2$ such that $f\circ c_1=c_2$. Clearly, this defines an equivalence
relation  in  the class of all
extensions of $X$; the equivalence class of an extension $(Y,c)$
of $X$ will be denoted by $[(Y,c)].$ We write $$(Y_1,c_1)\le
(Y_2,c_2) $$ and say that the extension $(Y_2,c_2)$ is {\em
 larger} than the extension $(Y_1,c_1)$ if there
exists a continuous mapping $f:Y_2\lra Y_1$ such that $f\circ
c_2=c_1$. This relation is a {\em preorder}\/ (i.e., it is
reflexive and transitive).
 In the class of all Hausdorff extensions
of $X$, the equivalence relation associated with this  preorder
(i.e., $(Y_1,c_1)$ is larger than $(Y_2,c_2)$ and conversely)
coincide  with the relation of equivalence defined above.

Setting for every two Hausdorff extensions $(Y_i,c_i),\ i=1,2$, of
a Hausdorff space $X$, $$[(Y_1,c_1)]\le [(Y_2,c_2)] \mbox{ iff }
(Y_1,c_1)\le (Y_2,c_2),$$ we obtain a well-defined relation on the
set of all, up to equivalence, Hausdorff extensions of $X$; it is
already an order.
\end{defi}

\begin{defi}\label{admba}{\rm (Ph. Dwinger \cite{Dwinger})}
\rm Let $(X,\TT)$ be a zero-dimensional Hausdorff space.
 A Boolean algebra $A$ is called {\em admissible for}
$(X,\TT)$ (or, a {\em Boolean base for} $(X,\TT)$) if $A$ is a
Boolean subalgebra of the Boolean algebra $\co (X)$ and $A$ is an
open base for $(X,\TT)$.
 The set of all admissible Boolean algebras for $(X,\TT)$ will be
denoted by $\BB\AA(X,\TT)$ (or, simply, by $\BB\AA(X)$).
\end{defi}

\begin{nota}\label{cz}
\rm The set of all (up to equivalence) zero-dimensional compact
Hausdorff
  extensions of a zero-dimensional Hausdorff space $(X,\TT)$
will be denoted by $\KK_0(X,\TT)$ (or, simply, by  $\KK_0(X)$). The order on $\KK_0(X,\TT)$
induced by the order $``\le$ " on the set of all Hausdorff
extensions of $X$ (defined in \ref{niDE}) will be denoted again by
$``\le$".
\end{nota}

\begin{theorem}\label{dwingerth}{\rm (Ph. Dwinger \cite{Dwinger})}
Let $(X,\TT)$ be a zero-dimensional Hausdorff space. Then the
ordered sets  $(\BB\AA(X,\TT),\sbe)$ and $(\KK_0(X,\TT),\le)$
are isomorphic. The isomorphism $\d$ between these two ordered
sets is the following one: for every $A\in \BB\AA(X,\TT)$,
$\d(A)\df [(\SSS (A),e_A)]$, with $e_A: X\lra \SSS (A)$  defined by
$e_A(x)\df (\hat{x}:A\lra\2)$, for every $x\in X$
(see \ref{hats} for the notation $\hat{x}$).
\end{theorem}

For every zero-dimensional Hausdorff space $X$, the ordered set
$(\BB\AA(X),\sbe)$ has a greatest element, namely the Boolean
algebra $\co (X)$. Thus, by the Dwinger Theorem \ref{dwingerth},
the ordered set $(\KK_0(X),\le)$ also has a greatest element. It
is denoted by $(\b_0X,\b_0)$. This fact was discovered earlier by
B. Banaschewski \cite{Ba} and $(\b_0X,\b_0)$ is said to be {\em
the Banaschewski compactification of $X$}. Clearly,
$(\b_0X,\b_0)=\d(\co (X))$, i.e. $\b_0X=\SSS (\co (X))$ and
$\b_0=e_{\co  (X)}$.

\begin{theorem}\label{zdextcb}{\rm (B. Banaschewski \cite{Ba})}
Let $(X_i,\TT_i)$, $i=1,2$, be  zero-dimensional Hausdorff spaces
and $(cX_2,c)$ be a zero-dimensional Hausdorff compactification of
$X_2$. Then for every continuous function $f:X_1\lra X_2$ there
exists a continuous function $g:\b_0X_1\lra cX_2$ such that
$g\circ\b_0=c\circ f$.
\end{theorem}

\begin{nist}\label{tar}
\rm We will need the Tarski Duality between the categories $\Set$
and $\Caba$. It consists of two contravariant functors
$${\sf P}:\Set\lra\Caba \ \  \mbox{ and } \ \ {\sf At}:\Caba\lra\Set$$
which are defined as follows. For every set $X$, $${\sf P}(X)\df
(P(X),\sbe).$$ If $f\in\Set(X,Y)$, then ${\sf P}(f): {\sf P}(Y)\lra
{\sf P}(X)$ is defined by the formula $${\sf P}(f)(M)\df f\inv(M),$$ for
every $M\in P(Y)$. Further, for every $B\in|\Caba|$,
$${\sf At}(B)\df\At(B);$$ if $\s\in\Caba(B,B\ap)$, then
${\sf At}(\s):{\sf At}(B\ap)\lra {\sf At}(B)$ is defined by the formula
$${\sf At}(\s)(x\ap)\df \bigwedge\{b\in B\st x\ap\le\s(b)\},$$ for every
$x\ap\in \At(B\ap)$.

For each set $X$, we have a bijection
$\eta_X: X\lra {\sf At}({\sf P}(X)),$ given by $\eta_X(x)\df\{x\}$ for every
$x\in X$, and $$\eta:\Id_{\Set}\lra \At\circ\PPP,\ \  X\mapsto\eta_X,$$ is a natural isomorphism.

For each $B\in|\Caba|$ we have a $\Caba$-isomorphism
$$\ep_B: B\lra {\sf P}({\sf At}(B)),$$ given by $\ep_B(b)\df\{x\in
\At(B)\st x\le b\}$ for each $b\in B$, and $$\ep:\Id_{\Caba}\lra \PPP\circ\At,\ \  B\mapsto\ep_B,$$ is a natural isomorphism.  Note that
$\ep_B\inv(M)=\bigvee_B M$, for all $M\sbe \At(B)$.

Thus $\klam{\PPP,\Att,\eta,\ep}:\Set\lra\Caba$ is an adjoint dual equivalence.
\end{nist}

The following assertion is well known (because $\Att(\s)$ is the restriction to $\At(B\ap)$ of the lower (or, left) adjoint for $\s$ (see \cite[Theorem 4.2]{J})), but we will present here its short proof.

\begin{lm}\label{taroblm}
Let $\s\in\Caba(B,B\ap)$. Then, for every $b\in B$ and each $x\ap\in\At(B\ap)$, $(x\ap\le\s(b))\Leftrightarrow (\Att(\s)(x\ap)\le b)$.
\end{lm}

\begin{proof}
Since $\Att(\s)(x\ap)= \bigwedge\{b\in B\st x\ap\le\s(b)\},$
we obtain immediately that $(x\ap\le\s(b))\Rightarrow (\Att(\s)(x\ap)\le b)$.
Suppose now that $\Att(\s)(x\ap)\le b$. Then $\s(\Att(\s)(x\ap))\le \s(b)$.
Since $\s(\Att(\s)(x\ap))=\s(\bigwedge\{c\in B\st x\ap\le\s(c)\})=\bigwedge\{\s(c)\st c\in B, x\ap\le\s(c)\}\ge x\ap$, we obtain that $x\ap\le\s(b)$.
\end{proof}

\begin{nist}\label{RC}
\rm
A set $F$ in a topological space $X$ is {\em regular closed} (or a {\em closed domain} \cite{E}) if it is the closure of its interior in $X$: $F={\rm cl(int}(F))$. The collection $\RC(X)$ of all regular closed sets in $X$ becomes a Boolean algebra, with the Boolean operations $\vee,\we,\,^*,0,1$ given by
\begin{align*}
F\vee G &= F\cup G, & F\we G &= \cl(\int(F\cap G)), & F^* &=  \cl(X\stm F), & 0 &=  \emptyset, & 1 &=  X.
\end{align*}
The Boolean algebra ${\RC}(X)$ is actually complete, with the infinite joins and meets given by
\begin{align*}
 \bigvee_{i\in I}F_i &=  \cl(\bigcup_{i\in I}F_i)\
=\cl(\bigcup_{i\in I}\int(F_i))=\cl(\int(\bigcup_{i\in I}F_i)), \quad\quad\quad
\bigwedge_{i\in I}F_i &=  \cl(\int(\bigcap_{i\in I}F_i)).
 \end{align*}
 \end{nist}

We will need as well the following well-known statement  (see, e.g., \cite{CNG},
p.271, and, for a proof, \cite{VDDB}).

\begin{lm}\label{isombool}
Let $X$ be a dense subspace of a topological space $Y$. Then the
functions $$r:\RC(Y)\lra \RC(X),\  F\mapsto F\cap X,$$ and
$$e:\RC(X)\lra\RC(Y),\  G\mapsto \cl_Y(G),$$ are Boolean isomorphisms between Boolean
algebras $\RC(X)$ and $\RC(Y)$, and $e\circ r=id_{\RCO(Y)}$, $r\circ
e=id_{\RCO(X)}$. (We will sometimes write $r_{X,Y}$ (resp., $e_{X,Y}$) instead of $r$ (resp., $e$).)
\end{lm}

\section{The first duality theorem for the ca\-tegory $\ZH$}

\begin{defi}\label{zalg}
\rm A pair $(A,X)$, where $A$ is a Boolean algebra and $X\sbe
\Bool(A,\2)$, is called a {\em Boolean z-algebra} (briefly, {\em
z-algebra}; abbreviated as ZA) if for each $a\in A^+$ there exists
$x\in X$ such that $x(a)=1$.
\end{defi}

Using the definition of the space $\SSS (A)$ (see \ref{nistone}),
where $A$ is a Boolean algebra, we obtain immediately the
following result:

\begin{fact}\label{zalgf}
 A pair $(A,X)$ is a z-algebra \tiff $A$ is a Boolean algebra and $X$ is
 a dense subset of\/ $\SSS (A)$.
\end{fact}

\begin{nota}\label{zalgn}
\rm If $A$ is a Boolean algebra and $X\sbe \Bool(A,\2)$,  we set
$$s_A^X(a)\df X\cap s_A(a)$$ for each $a\in A$ (see (\ref{sofa})
for $s_A$), defining in such a way a map $$s_A^X: A\lra \PPP(X), \ \
a\mapsto s_A^X(a).$$
\end{nota}

\begin{fact}\label{zalgff}
 A pair $(A,X)$ is a z-algebra \tiff $A$ is a Boolean algebra, $X\sbe \Bool(A,\2)$ and
 $s_A^X:A\lra \PPP(X)$ is a Boolean monomorphism.
\end{fact}

\begin{proof}
Suppose that  $(A,X)$ is a ZA. Then, by Fact \ref{zalgf}, $X$ is a
dense subset of $K\df \SSS (A)$ and thus $\cl_K(s_A^X(a))=s_A(a)$
for each $a\in A$. Therefore, using the fact that $s_A$ is a
Boolean isomorphism, we obtain that $s_A^X$ is a Boolean
monomorphism.

Conversely, if $s_A^X$ is a Boolean monomorphism, then
$s_A^X(a)\nes$ for each $a\in A^+$. Thus $X$ is dense in $\SSS (A)$,
which implies that $(A,X)$ is a ZA.
\end{proof}

\begin{fact}\label{zalgfff}
 Let $(A,X)$ be a z-algebra. Then the subspace topology on $X$ induced by\/
 $\SSS (A)$ coincides with the topology on $X$ generated by the base
 $s_A^X(A)$ and $s_A^X(A)\sbe \CO(X)$.
\end{fact}

\begin{proof}
Set $K\df \SSS (A)$. Then $\CO(K)=s_A(A)$ and $\CO(K)$ is a base for
$K$. Regarding $X$ as a subspace of $K$ and using the fact that
$s_A^X(A)=X\cap s_A(A)$, we obtain that $s_A^X(A)$ is a base for
the subspace topology on $X$ induced by $K$ and $s_A^X(A)\sbe
\CO(X)$. Hence, the topology on $X$ generated by the base
 $s_A^X(A)$ coincides with the subspace topology on $X$ induced by
 $K$.
\end{proof}

{\em When $(A,X)$ is a z-algebra, having in mind Fact \ref{zalgfff}, we will denote by $\bar{s}_A^X$ the map $s_A^X$ regarded as a map from $A$ to $\co (X)$.}

\begin{defi}\label{dzalgn}
\rm A z-algebra $(A,X)$ is called a {\em Boolean dz-algebra}
(briefly, {\em dz-algebra}; abbreviated as DZA) if $s_A^X(A)=\CO(X)$.
\end{defi}

Now, using Fact \ref{zalgff}, we obtain immediately the following result:

\begin{fact}\label{dzalgfn}
 A z-algebra $(A,X)$ is a DZA \tiff the map $\bar{s}_A^X:A\lra\co (X)$ is a Boolean isomorphism
 (regarding $X$ as a subspace of\/ $\SSS (A)$).
\end{fact}

\begin{exa}\label{dzalge}
\rm Let $A$ be a Boolean algebra.  Then
$(A,\Bool(A,\2))$ is a dz-algebra. (The dz-algebras of this type will be called
{\em compact Boolean dz-algebras} (or, simply, {\em compact dz-algebras})

Indeed, setting $X_A\df \Bool(A,\2)$, we have that $(A,X_A)$ is a z-algebra, $\bar{s}_A^{X_A}=s_A$ and thus
$\bar{s}_A^{X_A}(A)=\CO(X_A)$.
Hence, $(A,X_A)$ is a DZA.
\end{exa}

\begin{exa}\label{dzalgez}
\rm   Let $X$ be a zero-dimensional Hausdorff space and $A\in\BB\AA(X)$ (see Definition \ref{admba}). Then the pair $(A,\hat{X}_A)$ is a z-algebra, the pair $(\co (X),\hat{X})$ is a dz-algebra and the map $\hat{h}_{X,A}:X\lra \hat{X}_A$ is a homeomorphism (see \ref{hats} for the notation).


Indeed, the pair $(A,\hat{X}_A)$ is a z-algebra since for every $U\in A^+$ there exists $x\in U$ and thus $\hat{x}(U)=1$.  Also, we have to show that $\hat{h}_{X,A}$ is a homeomorphism. The family $A$ $T_0$-separates the points of $X$ because $A$ is a base of the Hausdorff space $X$. Hence, by \ref{hats}, $\hat{h}_{X,A}$ is a bijection. The family $\hat{X}_A\cap\co(\SSS(A))=\hat{X}_A\cap s_A(A)=s_A^{\hat{X}_A}(A)$ is a base of $\hat{X}_A$ and, for every $U\in A$, $s_A^{\hat{X}_A}(U)=\{\hat{x}\in\hat{X}_A\st \hat{x}(U)=1\}=\{\hat{x}\in\hat{X}_A\st x\in U\}=\hat{h}_{X,A}(U)$; thus, $\hat{h}_{X,A}\inv(s_A^{\hat{X}_A}(U))=U$. This shows that $\hat{h}_{X,A}$ is a continuous and open bijection and, therefore, it is a homeomorphism. Finally, if $A=\co(X)$, then, since $\hat{h}_X=\hat{h}_{X,A}$ is a homeomorphism, $\hat{h}_X(\co(X))=\co(\hat{X})$. Thus, $s_A^{\hat{X}}(\co(X))=\CO(\hat{X})$, i.e., $(\co(X),\hat{X})$ is a DZA.

\end{exa}

\begin{exa}\label{dzalget}
\rm    The pair $(B,\check{X}_B)$, where $B\in|\Caba|$, is a dz-algebra (see \ref{hats} for the notation $\check{X}_B$).
(The dz-algebras of this type will be called
{\em Boolean T-algebras} (or, simply, {\em  T-algebras})).

Indeed, for every $b\in B^+$, there exists $x\in\At(B)$ such that $x\le b$. Then $\check{x}(b)=1$. Thus, $(B,\check{X}_B)$ is a z-algebra.
For every $x\in\At(B)$, we have that $s_B^{\check{X}_B}(x)=\{\check{x}\}$. Hence, $\check{X}_B$ is a discrete subspace of $\SSS(B)$. Therefore, $\CO(\check{X}_B)=P(\check{X}_B)$. By \ref{hats}, the function $\check{h}_B:\At(B)\lra \check{X}_B$, $x\mapsto\check{x}$, is a bijection. Also, if $M\sbe \At(B)$ and $b_M=\bigvee M$, then $M=\{x\in\At(B)\st x\le b_M\}$. Finally, for every $b\in B$, $s_B^{\check{X}_B}(b)=\{\check{x}\in \check{X}_B\st \check{x}(b)=1\}=\{\check{x}\in \check{X}_B\st x\le b\}$. Thus, $s_B^{\check{X}_B}(B)=P(\check{X}_B)$. This shows that $(B,\check{X}_B)$ is a dz-algebra.
\end{exa}

We will present an equivalent definition of the notion of dz-algebra as well.

\begin{defi}\label{teq}
\rm Let $C\in|\Caba|$  and $A,B$ be Boolean subalgebras of $C$. If
for  every $a\in A$ and any $x\in \At(C)$ such that $x\le a$ there
exists $b\in B$ with
 $x\le b\le a$, then we will say that {\em $A$ is
t-coarser than $B$ in $C$} or that {\em $B$ is t-finer than $A$}
in $C$; in this case we will write $A\preceq_C B$. We will say
that the Boolean algebras $A$ and $B$ are {\em t-equal in $C$} if
$A\preceq_C B$ and $B\preceq_C A$.
\end{defi}

The following assertion is obvious:

\begin{fact}\label{teqf}
Let $X$ be a set and $A,B$ be Boolean subalgebras of the Boolean
algebra\/ $\PPP(X)$. Let $\OO_A$ (resp., $\OO_B$) be the
topology on $X$ generated by the base $A$ (resp., $B$). Then $A$
and $B$ are t-equal in\/ $\PPP(X)$ \tiff the topologies $\OO_A$
and $\OO_B$ coincide.
\end{fact}

\begin{fact}\label{dzalgf1}
 A z-algebra $(A,X)$ is a DZA \tiff it satisfies
the following condition:

\smallskip

\noindent(Dw) \  If $B$ is a Boolean subalgebra of\/ $\PPP(X)$
and $B$ is t-equal to $s_A^X(A)$ in\/ $\PPP(X)$, then $B\sbe
s_A^X(A)$.
\end{fact}

\begin{proof}
Suppose that the ZA $(A,X)$ satisfies condition (Dw). By Fact
\ref{zalgfff}, we have that $s_A^X(A)$ is a base for $X$ and
$s_A^X(A)\sbe \CO(X)$.
 Then the Fact \ref{teqf} shows that the Boolean algebras
$s_A^X(A)$ and  $\co (X)$ are t-equal in $\PPP(X)$. Thus, by
condition (Dw), we obtain that $\CO(X)\sbe s_A^X(A)$. Therefore,
$s_A^X(A)=\CO(X)$, i.e., $(A,X)$ is a DZA.

Conversely, suppose that $(A,X)$ is a DZA. If $B$ is a Boolean
subalgebra of $\PPP(X)$ and $B$ is t-equal to $s_A^X(A)$ in
$\PPP(X)$, then $B\sbe \co (X)$. Therefore, $B\sbe s_A^X(A)$.
This shows that $(A,X)$ satisfies condition (Dw).
\end{proof}

Now, we are ready to formulate and prove our first duality theorem for the category $\ZH$.
The proof of the next assertion is obvious.

\begin{pro}\label{zboolpro}
There is a category $\DZA$ (resp., $\ZA$) whose objects are all dz-algebras (resp., z-algebras) and
whose morphisms between any two $\DZA$-objects (resp., $\ZA$-objects)  $(A,X)$ and
$(A\ap,X\ap)$ are all pairs $(\p,f)$ such that
$\p\in\Bool(A,A\ap)$, $f\in\Set(X\ap,X)$ and $x\ap\circ\p=f(x\ap)$
for every $x\ap\in X\ap$. The composition $(\p\ap,f\ap)\circ (\p,f)$ between two $\DZA$-morphisms (resp., $\ZA$-morphisms)
$(\p,f):(A,X)\lra(A\ap,X\ap)$ and
$(\p\ap,f\ap):(A\ap,X\ap)\lra (A'',X'')$ is defined to be the $\DZA$-morphism (resp., $\ZA$-morphism)
$(\p\ap\circ\p, f\circ f\ap):(A,X)\lra (A'',X'')$; the identity
morphism of a $\DZA\ap$-object (resp., $\ZA$-object) $(A,X)$ is defined to be $(id_A,id_X)$.
\end{pro}

\begin{theorem}\label{zduality}
The categories\/ $\ZH$ and\/ $\DZA$ are dually equivalent.
\end{theorem}

\begin{proof}
We will first define a contravariant functor $$F:\ZH\lra\DZA.$$

For every $X\in|\ZH|$, let
$$F(X)\df(\co (X),\hat{X}).$$  Then Example \ref{dzalgez} shows that
$F(X)\in|\DZA|$. Further, for  $f\in\ZH(X,Y)$, set
 $$F(f)\df (\co (f),\hat{f}),$$
 where
 $$\hat{f}:\hat{X}\lra \hat{Y}$$ is defined by
$$\hat{f}(\hat{x})\df\widehat{f(x)}$$ for every $x\in X$. We will show that
$F(f)\in\DZA(F(Y),F(X)).$
We need only to prove that $\hat{x}\circ\co(f)=\widehat{f(x)}$ for every $x\in X$. So, let
$x\in X$. Then, for every $U\in \CO(Y)$, we have that
$(\hat{x}\circ\co(f))(U)=1\Leftrightarrow\hat{x}(f\inv(U))=1\Leftrightarrow x\in f\inv(U)\Leftrightarrow
f(x)\in U\Leftrightarrow \widehat{f(x)}(U)=1$. Therefore,
$\hat{x}\circ\co(f)=\hat{f}(\hat{x})$, for every $x\in X$.  Thus, $F(f)\in\DZA(F(Y),F(X))$.

It is easy to see that $F$ is a contravariant functor.

Now we will define a contravariant functor $$G:\DZA\lra\ZH$$ and
will prove that the functors $F\circ G$ and $G\circ F$ are
naturally isomorphic to the corresponding identity functors.

For every $(A,X)\in|\DZA|$, we set $$G(A,X)\df X,$$ where $X$ is regarded as a subspace of
$\SSS (A)$. Then, clearly, $G(A,X)\in|\ZH|$. If
$(\p,f):(A,X)\lra (A\ap,X\ap)$
is a $\DZA$-morphism, we put
$$G(\p,f)\df f.$$
Let us show that
$G(\p,f)$ is a continuous function. We have that $X\ap\sbe
\SSS (A\ap)=\Bool(A\ap,\2)$ and $X\sbe \SSS (A)=\Bool(A,\2)$.
For every $x\ap\in X\ap$,
$\SSS (\p)(x\ap)=x\ap\circ\p=f(x\ap)$. Thus, $f$ is a restriction of the continuous function
$\SSS (\p)$. Hence, $f:X\ap\lra X$ is a continuous function.
Therefore, $G$ is well-defined. Now it is easy to see that $G$ is
a contravariant functor.

We will show that the functors $F\circ G$ and $\Id_{\DZA}$ are naturally isomorphic.

Let $(A,X)\in|\DZA|$. Then $F(G(A,X))=F(X)=(\co (X),\hat{X})$, where $X$ is regarded as a subspace of $\SSS(A)$. By Fact \ref{dzalgfn}, the map
$\bar{s}_A^X:A\lra\co (X)$ is a Boolean isomorphism.
We put ${\breve{\imath}}_X\df \hat{h}_X\inv$ (recall that, by \ref{hats}, $\hat{h}_X$ is a bijection). Hence,
 $${\breve{\imath}}_X:\hat{X}\lra X,\ \  \hat{x}\mapsto x,$$ for every $x\in X$.
 Also, for every $x\in X$, $\hat{x}\circ \bar{s}_A^X={\breve{\imath}}_X(\hat{x})$. Indeed, for every $a\in A$, $\hat{x}(\bar{s}_A^X(a))=1\Leftrightarrow x\in s_A(a)\Leftrightarrow x(a)=1$, and thus $\hat{x}\circ \bar{s}_A^X=x={\breve{\imath}}_X(\hat{x})$. This shows that the map $(\bar{s}_A^X,{\breve{\imath}}_X):(A,X)\lra (\co (X),\hat{X})$ is a $\DZA$-morphism and, moreover, it is a $\DZA$-isomorphism. We put $s\ap_{(A,X)}\df (\bar{s}_A^X,{\breve{\imath}}_X)$. Then $$s\ap_{(A,X)}:(A,X)\lra (F\circ G)(A,X)$$ is a $\DZA$-isomorphism. Let $(\p,f):(A,X)\lra(A\ap,X\ap)$ be a $\DZA$-morphism. We will show that the diagram
\begin{center}
$\xymatrix{(A,X)\ar[rr]^{(\p,f)}\ar[d]_{s\ap_{(A,X)}} && (A\ap,X\ap)\ar[d]^{s\ap_{(A\ap,X\ap)}}\\
(\co(X),\hat{X})\ar[rr]^{(F\circ G)(\p,f)\;} && (\co(X\ap),\widehat{X\ap})
}$
\end{center}
is commutative. Indeed, we have that $$s\ap_{(A\ap,X\ap)}\circ (\p,f)=(\bar{s}_{A\ap}^{X\ap},{\breve{\imath}}_{X\ap})\circ (\p,f)=(\bar{s}_{A\ap}^{X\ap}\circ \p, f\circ {\breve{\imath}}_{X\ap})$$ and $$(F\circ G)(\p,f)\circ s\ap_{(A,X)}=(\co(f),\hat{f})\circ (\bar{s}_A^X,{\breve{\imath}}_X)=(\co(f)\circ \bar{s}_A^X,{\breve{\imath}}_X\circ\hat{f}).$$ Thus, we have to show that
$$\bar{s}_{A\ap}^{X\ap}\circ \p=\co(f)\circ \bar{s}_A^X\ \ \mbox{ and }\ \ f\circ {\breve{\imath}}_{X\ap}={\breve{\imath}}_X\circ\hat{f}.$$ Let $a\in A$. Then

\smallskip

\begin{tabular}{ll}
$\co(f)(\bar{s}_A^X(a))$ & $=f\inv(X\cap s_A(a))$\\
&$=\{x\ap\in X\ap\st f(x\ap)\in s_A(a)\}$\\
&$=\{x\ap\in X\ap\st f(x\ap)(a)=1\}$\\
&$=\{x\ap\in X\ap\st x\ap(\p(a))=1\}$\\
&$=X\ap\cap s_{A\ap}(\p(a))$\\
&$=\bar{s}_{A\ap}^{X\ap}(\p(a)).$\\
\end{tabular}

Also, for every $x\ap\in X\ap$, ${\breve{\imath}}_X(\hat{f}(\widehat{x\ap}))={\breve{\imath}}_X(\widehat{f(x\ap)})=f(x\ap)
=f({\breve{\imath}}_{X\ap}(\widehat{x\ap}))$.
Hence, $$s\ap: \Id_{\DZA}\lra F\circ G,\ \ (A,X)\mapsto s\ap_{(A,X)},$$ is a natural isomorphism.

Finally, we will show that the functors $G\circ F$ and $\Id_{\ZH}$ are naturally isomorphic.
Let $X\in|\ZH|$. Then $G(F(X))=G(\co (X),\hat{X})=\hat{X}$,
where $\hat{X}$ is regarded as a subspace of $\SSS (\co (X))$.
By Example \ref{dzalgez},  $\hat{h}_X$ is a homeomorphism.
Let $f:X\lra Y$ be a $\ZH$-morphism. Then $G(F(f))=\hat{f}$, and we have to show that the diagram
\begin{center}
$\xymatrix{X\ar[rr]^{f}\ar[d]_{\hat{h}_X} && Y\ar[d]^{\hat{h}_Y}\\
\hat{X}\ar[rr]^{\hat{f}} && \hat{Y}
}$
\end{center}
is commutative. For every $x\in X$, we have $\hat{h}_Y(f(x))=\widehat{f(x)}=\hat{f}(\hat{x})=\hat{f}(\hat{h}_X(x))$.
Therefore, $$\hat{h}:\Id_{\ZH}\lra G\circ F,\ \ X\mapsto \hat{h}_X,$$
is a natural isomorphism. All this shows that the categories\/ $\ZH$ and\/ $\DZA$ are dually equivalent.
\end{proof}

\section{The second duality theorem for the ca\-tegory $\ZH$}

Now we will define a new category $\MBool$ and will show, using the
Tarski duality, that it is equivalent to the category $\DZA$.
This will imply immediately that the category $\MBool$ is dually
equivalent to the category $\ZH$. The category $\MBool$ is similar
to the category $\MDeVe$, constructed in \cite{BMO} as a category
dually equivalent to the category $\Tych$ of Tychonoff spaces and
continuous maps.

\begin{defi}\label{zhomo}
\rm Let $A$ be a Boolean algebra and $B\in|\Caba|$. A Boolean
monomorphism $\a:A\lra B$ is said to be a {\em Boolean z-map}
(briefly, {\em z-map}) if every atom of $B$ is a meet of some
elements of $\a(A)$. A z-map $\a:A\lra B$ is called a {\em maximal
Boolean z-map} (briefly, {\em mz-map}) if $\CO(X_\a)=s_A^{X_\a}(A)$, where $X_\a$ is regarded as a subspace of $\SSS(A)$
(see \ref{hats} and \ref{zalgn} for the notation).
\end{defi}

\begin{exa}\label{zmapt}
\rm  Let $X\in|\ZH|$, $A\in\BB\AA(X)$ (see Definition \ref{admba} for this notation) and $i_A:A\hookrightarrow \PPP(X)$ be the inclusion monomorphism. Then $i_A$ is a z-map.

Indeed, since $A$ is a base for the Hausdorff space $X$, we have that for every $x\in X$, $\{x\}=\bigcap\{U\in A\st x\in U\}$.  Hence, $i_A$ is a z-map.
\end{exa}

\begin{exa}\label{mzmapt}
\rm    The map $id_B:B\lra B$, $b\mapsto b$, where $B\in|\Caba|$, is a mz-map.
(The mz-maps of this type will be called
{\em Boolean T-maps} (or, simply, {\em  T-maps})).

Indeed, it is obvious that $id_B$ is a z-map. Setting $\a\df id_B$, we obtain, as in \ref{hats}, that $X_\a=\check{X}_B$. In Example \ref{dzalget}, we proved that $\check{X}_B$ is a discrete subspace of $\SSS(B)$ (and, thus,  $\CO(\check{X}_B)=P(\check{X}_B)$) and $s_B^{\check{X}_B}(B)=P(\check{X}_B)$. This shows that $id_B$ is a mz-map.
\end{exa}

\begin{exa}\label{mzmapek}
\rm    The map $s_A^{X_A}:A\lra \PPP(X_A)$, where $A\in|\Bool|$ and $X_A=\Bool(A,\2)$, is a mz-map.
(The mz-maps of this type will be called
{\em compact mz-maps}.)

Indeed, since $s_A^{X_A}\upharpoonright A=s_A:A\lra \co(X_A)$, $X_A=\SSS(A)$ is a Hausdorff space and $\co(X_A)$ is a base for $X_A$, we obtain that $s_A^{X_A}$ is a z-map. Set $\a\df s_A^{X_A}$. Then $X_\a=\{\a_x:A\lra\2\st x\in X_A\}$ and, for every $x\in X_A$ and every $a\in A$, $\a_x(a)=1\Leftrightarrow x\in\a(a)\Leftrightarrow x(a)=1$. Thus, $\a_x\equiv x$. Hence, $X_\a=X_A$. Then $s_A^{X_\a}(A)= s_A^{X_A}(A)=s_A(A)=\co(X_A)=\co(X_\a)$. Therefore, $s_A^{X_A}$ is a mz-map.
\end{exa}

We will present an equivalent definition of the notion of mz-map as well.
Its straightforward proof is left to the reader.

\begin{pro}\label{zhomo1}
\rm Let $A$ be a Boolean algebra and $B\in|\Caba|$.  A z-map $\a:A\lra B$ is  an mz-map \tiff for every Boolean
subalgebra $C$ of $B$ which is t-equal to $\a(A)$ in $B$, we
have that $C\sbe \a(A)$.
\end{pro}

The proof of the next assertion is obvious.

\begin{pro}\label{mboolpro}
There is a category $\MBool$ (resp., $\ZBool$) whose objects are all mz-maps (resp., z-maps) and
whose morphisms between any two $\MBool$-objects (resp., $\ZBool$-objects) $\a:A\lra B$ and
$\a\ap:A\ap\lra B\ap$ are all pairs $(\p,\s)$ such that
 $\p\in\Bool(A,A\ap)$, $\s\in\Caba(B,B\ap)$ and
$\a\ap\circ\p=\s\circ\a$. The composition $(\p\ap,\s\ap)\circ (\p,\s)$ between two $\MBool$-morphisms
(resp., $\ZBool$-morphisms)
$(\p,\s):\a\lra\a\ap$ and
$(\p\ap,\s\ap):\a\ap\lra \a''$ is defined to be the $\MBool$-morphism (resp., $\ZBool$-morphism)
$(\p\ap\circ\p, \s\ap\circ \s):\a\lra \a''$; the identity map of
an $\MBool$-object (resp., $\ZBool$-object) $\a:A\lra B$ is defined to be $(id_A,id_B)$.
\end{pro}

\begin{theorem}\label{zeq}
The categories\/ $\MBool$ and\/ $\DZA$ are equivalent.
\end{theorem}

\begin{proof}
We start by defining a functor $F\ap:\DZA\lra\MBool$.

For every $(A,X)\in|\DZA|$, set $$F\ap(A,X)\df s_A^X$$ (see Notation
\ref{zalgn} for $s_A^X$). For showing that $F\ap(A,X)\in|\MBool|$,
notice first that $\PPP(X)\in|\Caba|$ and,
by Fact \ref{zalgff},
$s_A^X$ is a Boolean monomorphism.
Furthermore, by Fact \ref{zalgfff}, the topology on $X$ generated
by the base $s_A^X(A)$ is a $T_2$-topology. Thus, for every $x\in
X$, we have that $\{x\}=\bigcap\{s_A^X(a)\st x(a)=1\}$. Hence,
$F\ap(A,X)$ is a z-map. Set $\a\df s_A^X$ and $B\df \PPP(X)$. Then $\a:A\lra B$ and $\At(B)=X$. Since $(A,X)\in|\DZA|$, we have that $\a(A)=\CO(X)$.
Using the notation from \ref{hats}, we obtain that for every $x\in X=\At(B)$ and every $a\in A$, $\a_x(a)=1\Leftrightarrow x\le\a(a)\Leftrightarrow x\in s_A(a)\Leftrightarrow x(a)=1$. Thus, $x=\a_x$ for every $x\in X$. Hence $X\equiv X_\a$ and, therefore, $s_A^{X_\a}(A)=s_A^X(A)=\CO(X)=\CO(X_\a)$.
This shows that $F\ap(A,X)\in|\MBool|$.

For every $(\p,f)\in\DZA((A,X),(A\ap,X\ap))$, set $$F\ap(\p,f)\df
(\p,\PPP(f)).$$
We have that $x\ap\circ\p=f(x\ap)$ for every $x\ap\in X\ap$.
Having this in mind, we obtain that for every $a\in A$, $(\PPP(f)\circ
s_A^X)(a)=
f\inv(\{x\in X\st x(a)=1\})=\{x\ap\in
X\ap\st f(x\ap)(a)=1\}=\{x\ap\in X\ap\st
(x\ap\circ\p)(a)=1\}=
(s_{A\ap}^{X\ap}\circ\p)(a)$.
 Hence, $\PPP(f)\circ
s_A^X=s_{A\ap}^{X\ap}\circ\p$. Since
$\PPP(f)\in\Caba(\PPP(X),\PPP(X\ap))$, we obtain that
$F\ap(\p,f)\in\MBool(F\ap(A,X),F\ap(A\ap,X\ap))$.
Now it is easy to see
that $F\ap$ is a functor.

Further, we will define a functor $G\ap:\MBool\lra \DZA$.

For every $(\a:A\lra B)\in |\MBool|$, we set, in the notation from \ref{hats},
$$G\ap(\a)\df(\co (X_\a),\widehat{X_\a}),$$
where $X_\a$ is regarded as a subspace of $\SSS(A)$.
Hence $X_\a=\{\a_x:A\lra\2\st x\in \At(B)\}$ and $\widehat{X_\a}=\{\widehat{\a_x}:\co (X_\a)\lra\2\st \a_x\in X_\a\}$.
Obviously,  $X_\a\in|\ZH|$. It is now clear  that $G\ap(\a)=F(X_\a)$ (where $F$ is the contravariant functor defined in the proof of Theorem \ref{zduality}) and, therefore, by Theorem \ref{zduality}, $G\ap(\a)\in |\DZA|$.

Let $(\p,\s)\in\MBool(\a,\a\ap)$, where $\a:A\lra B$ and $\a\ap:A\ap\lra B\ap$. We set
$$G\ap(\p,\s)\df (\co(f_\s),\widehat{f_\s}),$$
where $f_\s:X_{\a\ap}\lra X_\a$ is defined by $\a\ap_{x\ap}\mapsto \a_{\Att(\s)(x\ap)}$ and $\widehat{f_\s}:\widehat{X_{\a\ap}}\lra\widehat{X_\a}$ is defined by $\widehat{\a\ap_{x\ap}}\mapsto\widehat{f_\s(\a\ap_{x\ap})}$.
 Clearly, $G\ap(\p,\s)=F(f_\s)$, so that we need only to show that $f_\s$ is a continuous map between the sets $X_{\a\ap}$ and $X_\a$ supplied with the subspace topology from the spaces $\SSS (A\ap)$ and $\SSS (A)$, respectively. Let $a\in A$. Then, using Lemma \ref{taroblm}, we obtain that

\smallskip

\begin{tabular}{ll}
$f_\s\inv(X_\a\cap s_A(X))$ & $=\{\a\ap_{x\ap}\in X_{\a\ap}\st f_\s(\a\ap_{x\ap})(a)=1\}$\\
&$=\{\a\ap_{x\ap}\in X_{\a\ap}\st \a_{\Att(\s)(x\ap)}(a)=1\}$\\
&$=\{\a\ap_{x\ap}\in X_{\a\ap}\st \Att(\s)(x\ap)\le\a(a)\}$\\
&$=\{\a\ap_{x\ap}\in X_{\a\ap}\st x\ap\le\s(\a(a))\}$\\
&$=\{\a\ap_{x\ap}\in X_{\a\ap}\st x\ap\le\a\ap(\p(a))\}$\\
&$=\{\a\ap_{x\ap}\in X_{\a\ap}\st \a\ap_{x\ap}(\p(a))=1\}$\\
&$=X_{\a\ap}\cap s_{A\ap}(\p(a)).$\\
\end{tabular}

\noindent This implies the continuity of $f_\s$. Now, using Theorem \ref{zduality}, we conclude that $G\ap(\p,\s)\in \DZA(G\ap(\a),G\ap(\a\ap))$. Having all this in mind, it is easy to see
that $G\ap$ is a functor.

We will now prove that $F\ap\circ G\ap\cong \Id_{\MBool}$.

Let $(\a:A\lra B)\in|\MBool|$. Then $(F\ap\circ G\ap)(\a)=F\ap(\co (X_\a),\widehat{X_\a})=s_{\CLO(X_\a)}^{\widehat{X_\a}}$ and $s_{\CLO(X_\a)}^{\widehat{X_\a}}:\co (X_\a)\lra\PPP(\widehat{X_\a})$, where $X_\a$ is regarded as a subspace of $\SSS(A)$.
By \ref{hats},
the map $h_\a:\At(B)\lra X_\a$, $x\mapsto \a_x$, is a bijection.
Then, clearly, the map $h_\a^P:\PPP(\At(B))\lra\PPP(X_\a)$, $M\mapsto \{h_\a(x)\st x\in M\}$, is a Boolean isomorphism. Again
by \ref{hats}, the map $\hat{h}_{X_\a}: X_\a\lra\widehat{X_\a}$, $\a_x\mapsto\widehat{\a_x}$, for all $x\in\At(B)$, is a bijection.
 Then the map $\hat{h}_{X_\a}^P:\PPP(X_\a)\lra\PPP(\widehat{X_\a})$, $M\mapsto \{\hat{h}_{X_\a}(\a_x)\st \a_x\in M\}$, is a Boolean isomorphism. Put $\widehat{\ep_B}\df \hat{h}_{X_\a}^P\circ h_\a^P\circ\ep_B$ (see \ref{tar} for the notation $\ep_B$). Then
$$\widehat{\ep_B}:B\lra\PPP(\widehat{X_\a}),\ \ \ b\mapsto \{\widehat{\a_x}\st x\in\At(B), x\le b\}$$
is a Boolean isomorphism. Since $\a$ is an mz-map, we have that $s_A^{X_\a}(A)=\CO(X_\a)$. Thus the map
$$\bar{s}_A^{X_\a}:A\lra\co (X_\a),\ \  a\mapsto X_\a\cap s_A(a),$$ is a Boolean isomorphism. Put $$\ep\ap_\a\df(\bar{s}_A^{X_\a},\widehat{\ep_B}).$$ We will show that $\ep\ap_\a\in\MBool(\a,(F\ap\circ G\ap)(\a))$. We need only to prove that the diagram
\begin{center}
$\xymatrix{A\ar[rr]^{\bar{s}_A^{X_\a}}\ar[d]_{\a} && \co(X_\a)\ar[d]^{s_{\CLO(X_\a)}^{\widehat{X_\a}}}\\
B\ar[rr]^{\widehat{\ep_B}} && {\sf P}(\widehat{X_{\a}})
}$
\end{center}
is commutative. Let $a\in A$. Then $s_{\CLO(X_\a)}^{\widehat{X_\a}}(s_A^{\widehat{X_\a}}(a))=s_{\CLO(X_\a)}^{\widehat{X_\a}}(\{\a_y\in X_\a\st \a_y(a)=1\})=\{\widehat{\a_x}\in \widehat{X_\a}\st \widehat{\a_x}(\{\a_y\in X_\a\st y\le\a(a)\})=1\}=
\{\widehat{\a_x}\in \widehat{X_\a}\st \a_x\in \{\a_y\in X_\a\st y\le\a(a)\}\}=
\{\widehat{\a_x}\in \widehat{X_\a}\st x\le\a(a)\}=\widehat{\ep_B}(\a(a))$.
Obviously, this implies that $\ep\ap_\a$ is an $\MBool$-isomorphism.

Let $\a:A\lra B$ and $\a\ap:A\ap\lra B\ap$ be $\MBool$-objects and $(\p,\s):\a\lra\a\ap$ be an $\MBool$-morphism. We will prove that the diagram
\begin{center}
$\xymatrix{\a\ar[rr]^{(\p,\s)}\ar[d]_{\ep\ap_\a} && \a\ap\ar[d]^{\ep\ap_{\a\ap}}\\
s_{\CLO(X_\a)}^{\widehat{X_\a}}\ar[rr]^{F\ap(G\ap(\p,\s))\;} && s_{\CLO(X_{\a\ap})}^{\widehat{X_{\a\ap}}}
}$
\end{center}
is commutative.
We have that $$F\ap(G\ap(\p,\s))=F\ap(\co(f_\s),\widehat{f_\s})=(\co(f_\s),\PPP(\widehat{f_\s}))$$
and $$\ep\ap_{\a\ap}\circ(\p,\s)=(\bar{s}_{A\ap}^{X_{\a\ap}},\widehat{\ep_{B\ap}})\circ(\p,\s)=
(\bar{s}_{A\ap}^{X_{\a\ap}}\circ\p,\widehat{\ep_{B\ap}}\circ\s).$$
Also $$F\ap(G\ap(\p,\s))\circ \ep\ap_\a=
(\co(f_\s),\PPP(\widehat{f_\s}))\circ (\bar{s}_A^{X_\a},\widehat{\ep_B})=
(\co(f_\s)\circ \bar{s}_A^{X_\a},\PPP(\widehat{f_\s})\circ \widehat{\ep_B}).$$
Hence, we have to show that $$\bar{s}_{A\ap}^{X_{\a\ap}}\circ\p=\co(f_\s)\circ \bar{s}_A^{X_\a}\ \
\mbox{ and }\ \ \widehat{\ep_{B\ap}}\circ\s=\PPP(\widehat{f_\s})\circ \widehat{\ep_B}.$$
Let $a\in A$. Then, using Lemma \ref{taroblm}, we obtain that

\smallskip

\begin{tabular}{ll}
$\co(f_\s)(\bar{s}_A^{X_\a}(a))$ & $=f_\s\inv(\{\a_x\in X_\a\st \a_x(a)=1\})$\\
&$=\{\a\ap_{x\ap}\in X_{\a\ap}\st f_\s(\a\ap_{x\ap})\in \{\a_x\in X_\a\st \a_x(a)=1\}\}$\\
&$=\{\a\ap_{x\ap}\in X_{\a\ap}\st \a_{\Att(\s)(x\ap)}(a)=1\}$\\
&$=\{\a\ap_{x\ap}\in X_{\a\ap}\st \Att(\s)(x\ap)\le\a(a)\}$\\
&$=\{\a\ap_{x\ap}\in X_{\a\ap}\st x\ap\le(\s\circ\a)(a)\}$\\
&$=\{\a\ap_{x\ap}\in X_{\a\ap}\st  x\ap\le\a\ap(\p(a))\}$\\
&$=\{\a\ap_{x\ap}\in X_{\a\ap}\st \a\ap_{x\ap}(\p(a))=1\}$\\
&$=\bar{s}_{A\ap}^{X_{\a\ap}}(\p(a)).$\\
\end{tabular}

\noindent So, $\bar{s}_{A\ap}^{X_{\a\ap}}\circ\p=\co(f_\s)\circ \bar{s}_A^{X_\a}$.
Let now $b\in B$. Then, using again Lemma \ref{taroblm}, we obtain that

\smallskip

\begin{tabular}{ll}
$\PPP(\widehat{f_\s})(\widehat{\ep_B}(b))$ & $=\widehat{f_\s}\inv(\{\widehat{\a_x}\in \widehat{X_\a}\st x\le b\}$\\
&$=\{\widehat{\a\ap_{x\ap}}\in \widehat{X_{\a\ap}}\st \widehat{f_\s}(x\ap)\in \{\widehat{\a_x}\in \widehat{X_\a}\st x\le b\}\}$\\
&$=\{\widehat{\a\ap_{x\ap}}\in \widehat{X_{\a\ap}}\st \widehat{f_\s(\a\ap_{x\ap})}\in \{\widehat{\a_x}\in \widehat{X_\a}\st x\le b\}\}$\\
&$=\{\widehat{\a\ap_{x\ap}}\in \widehat{X_{\a\ap}}\st \Att(\s)(x\ap)\le b\}$\\
&$=\{\widehat{\a\ap_{x\ap}}\in \widehat{X_{\a\ap}}\st x\ap\le\s(b)\}=\widehat{\ep_{B\ap}}(\s(b)).$\\
\end{tabular}

\noindent Hence, $\widehat{\ep_{B\ap}}\circ\s=\PPP(\widehat{f_\s})\circ \widehat{\ep_B}$.
This shows that
$\ep\ap_{\a\ap}\circ(\p,\s)=F\ap(G\ap(\p,\s))\circ \ep\ap_\a$. Therefore, $$\ep\ap:\Id_{\MBool}\lra F\ap\circ G\ap,\ \ \a\mapsto\ep\ap_\a,$$
is a natural isomorphism.

Finally, we will prove that $G\ap\circ F\ap\cong \Id_{\DZA}$.

Let $(A,X)\in |\DZA|$. Then  $G\ap(F\ap(A,X))=G\ap(s_A^X)=(\co (X),\hat{X})$, where $X$ is regarded as asubspace of $\SSS(A)$. Indeed, putting $\a\df s_A^X$, we obtain, as in the beginning of this proof, that $\a_x\equiv x$ for every $x\in X$, and, hence,
$$X_\a\equiv X.$$ Thus, $\hat{x}:\co (X)\lra \2$ is defined by $\hat{x}(U)=1\Leftrightarrow x\in U$, for $U\in\CO(X)$, and $\hat{X}=\{\hat{x}\st x\in X\}$. Obviously, we have that $G\ap(F\ap(A,X))=F(G(A,X))$, where $F$ and $G$ are the contravariant functors defined in the proof of Theorem \ref{zduality}. Hence, we can use the $\DZA$-isomorphism $$s\ap_{(A,X)}:(A,X)\lra G\ap(F\ap(A,X))$$ defined there by $s\ap_{(A,X)}\df (\bar{s}_A^X,{\breve{\imath}}_X)$, where
${\breve{\imath}}_X:\hat{X}\lra X$,  $\hat{x}\mapsto x$.

Let $(\p,f)\in\DZA((A,X),(A\ap,X\ap))$. Then $G\ap(F\ap(\p,f))=G\ap(\p,\PPP(f))=(\co(f_\s),\widehat{f_\s})$, where $\s\df\PPP(f)$, $f_\s:X_{\a\ap}\lra X_\a$, $\a=F\ap(A,X)=s_A^X$ and $\a\ap=F\ap(A\ap,X\ap)=s_{A\ap}^{X\ap}$. Since $X_{\a\ap}\equiv X\ap$ and $X_\a\equiv X$ (see the beginning of this proof), we obtain that $f_\s(x\ap)=\Att(\s)(x\ap)=\Att(\PPP(f))(x\ap)=f(x\ap)$, i.e., $f_\s\equiv f$. Thus $G\ap(F\ap(\p,f))=(\co(f),\hat{f})=F(G(\p,f))$. Thus the proof of the commutativity of the diagram
\begin{center}
$\xymatrix{(A,X)\ar[rr]^{(\p,f)}\ar[d]_{s\ap_{(A,X)}} && (A\ap,X\ap)\ar[d]^{s\ap_{(A\ap,X\ap)}}\\
G\ap(F\ap(A,X))\ar[rr]^{(G\ap\circ F\ap)(\p,f)\;} && G\ap(F\ap(A\ap,X\ap))
}$
\end{center}
proceeds as in the proof of Theorem \ref{zduality}. Therefore,
$$s\ap: \Id_{\DZA}\lra G\ap\circ F\ap,\ \ (A,X)\mapsto s\ap_{(A,X)},$$ is a natural isomorphism.

All this shows that the categories\/ $\MBool$ and\/ $\DZA$ are equivalent.
\end{proof}

Obviously, Theorems \ref{zduality} and \ref{zeq} imply the
following theorem:

\begin{theorem}\label{nzduality}
The categories\/ $\ZH$ and\/ $\MBool$ are dually equivalent.
\end{theorem}

\begin{proof}
We put $\FFF_0\df F\ap\circ F$ and $\GGG_0\df G\circ G\ap$. Then $$\FFF_0:\ZH\lra \MBool \ \mbox{ and }\ \GGG_0:\MBool\lra\ZH.$$
Clearly, they are dual equivalences.
In the rest of this proof, we will find the explicit descriptions of these contravariant functors, as well as the descriptions of the natural isomorphisms $\tilde{\eta}^{\, 0}:\Id_{\ZH}\lra \GGG_0\circ\FFF_0$ and $\tilde{\ep}^{\, 0}:\Id_{\MBool}\lra \FFF_0\circ\GGG_0$. Moreover, we will define two new contravariant functors
$$\FFF:\ZH\lra \MBool \ \mbox{ and }\ \GGG:\MBool\lra\ZH$$
which are simpler than $\FFF_0$ and $\GGG_0$ but are again dual equivalences.

For every $X\in|\ZH|$, we have that $$\FFF_0(X)=F\ap(\co (X),\hat{X})=s_{\CLO(X)}^{\hat{X}}.$$

For every $f\in\ZH(X,Y)$, $$\FFF_0(f)=F\ap(\co(f),\hat{f})=(\co(f),\PPP(\hat{f}))$$
(see the beginning of the proof of Theorem \ref{zduality} for the notation $\hat{f}$).

For every $(\a:A\lra B)\in|\MBool|$, $$\GGG_0(\a)=G(\co (X_\a),\widehat{X_\a})=\widehat{X_\a},$$ where $X_\a$ is regarded as a subspace of $\SSS(A)$.

For $(\p,\s)\in\MBool(\a,\a\ap)$, 
$$\GGG_0(\p,\s)=G(\co(f_\s),\widehat{f_\s})=\widehat{f_\s}$$
(see the definition of $G\ap$ in the proof of Theorem \ref{zeq} for the notation $f_\s$ and $\widehat{f_\s}$).

Now, for every $X\in|\ZH|$, we have that
$$(\GGG_0\circ\FFF_0)(X)=\GGG_0(s_{\CLO(X)}^{\hat{X}})=\widehat{X_\a}, \mbox{ where } \a=s_{\CLO(X)}^{\hat{X}}.$$
Hence $X_\a=\{\a_{\hat{x}}:\co (X)\lra\2\st \hat{x}\in\hat{X}\}$ and, for every $U\in\CO(X)$ and every $x\in X$, $\a_{\hat{x}}(U)=1\Leftrightarrow \hat{x}\in\a(U)\Leftrightarrow \hat{x}(U)=1$. Thus, $\a_{\hat{x}}=\hat{x}$ for every $x\in X$. Hence, $X_\a=\hat{X}$ and $(\GGG_0\circ\FFF_0)(X)=\widehat{X_\a}=\hat{\hat{X}}$.
According to the general theorem about compositions of adjoint functors (see, e.g., \cite[Theorem IV.8.1]{MacLane}), we have that for every $X\in|\ZH|$, $\tilde{\eta}^{\, 0}_X:X\lra (\GGG_0\circ\FFF_0)(X)$ is defined by the formula $\tilde{\eta}^{\, 0}_X=(G(s\ap_{F(X)}))\inv\circ \hat{h}_X$
(see Theorem \ref{zduality} for $\hat{h}$ and $s\ap$). Since $F(X)=(\co (X),\hat{X})$ and $s\ap_{F(X)}=(\bar{s}_{\CLO(X)}^{\hat{X}},\breve{\imath}_{\hat{X}})$, where $\breve{\imath}_{\hat{X}}:\hat{\hat{X}}\lra\hat{X}$, $\hat{\hat{x}}\mapsto\hat{x}$, we obtain that $G(s\ap_{F(X)})=\breve{\imath}_{\hat{X}}$. Thus, $$\tilde{\eta}^{\, 0}_X(x)=\{\hat{\hat{x}}\},$$ for every $x\in X$.
Finally, note that $\CO(\hat{X})=\a(\CO(X))$ (because $\a=s_{\CLO(X)}^{\hat{X}}$ is an mz-map) and thus $\hat{\hat{x}}:\co (\hat{X})\lra\2$ is defined by $\hat{\hat{x}}(\a(U))=1\Leftrightarrow \hat{x}\in s_{\CLO(X)}^{\hat{X}}(U)\Leftrightarrow \hat{x}(U)=1\Leftrightarrow x\in U$, for every $x\in X$ and every $U\in\CO(X)$.

We will now describe the natural isomorphism $\tilde{\ep}^{\, 0}:\Id_{\MBool}\lra \FFF_0\circ\GGG_0$. For $(\a:A\lra B)\in|\MBool|$, we have that
$$(\FFF_0\circ\GGG_0)(\a)=\FFF_0(\widehat{X_\a})=s_{\CLO(\widehat{X_\a})}^{\widehat{\widehat{X_\a}}}$$
and $s_{\CLO(\widehat{X_\a})}^{\widehat{\widehat{X_\a}}}:\co (\widehat{X_\a})\lra\PPP(\widehat{\widehat{X_\a}})$,
where
$\widehat{\widehat{X_\a}}=\{\widehat{\widehat{\a_x}}:\co (\widehat{X_\a})\lra\2\st \widehat{\a_x}\in \widehat{X_\a}\}$ and, for $U\in\CO(\widehat{X_\a})$, $\widehat{\widehat{\a_x}}(U)=1\Leftrightarrow \widehat{\a_x}\in U$.
Thus $\tilde{\ep}^{\, 0}_\a:\a\lra s_{\CLO(\widehat{X_\a})}^{\widehat{\widehat{X_\a}}}$. The cited above theorem about compositions of adjoint functors gives us that $\tilde{\ep}^{\, 0}_\a=F\ap(s\ap_{G\ap(\a)})\circ\ep\ap_\a$.
We have that $G\ap(\a)=(\co (X_\a),\widehat{X_\a})$ and thus $s\ap_{G\ap(\a)}=(\bar{s}_{\CLO(X_\a)}^{\widehat{X_\a}},\breve{\imath}_{\widehat{X_\a}})$,
where $$\breve{\imath}_{\widehat{X_\a}}:\widehat{\widehat{X_\a}}\lra \widehat{X_\a},\ \ \widehat{\widehat{\a_x}}\mapsto \widehat{\a_x}, \ \mbox{ and } \bar{s}_{\CLO(X_\a)}^{\widehat{X_\a}}:\CLO(X_\a)\lra\PPP(\widehat{X_\a}).$$ Then
$F\ap(s\ap_{G\ap(\a)})=(\bar{s}_{\CLO(X_\a)}^{\widehat{X_\a}},\PPP(\breve{\imath}_{\widehat{X_\a}}))$. Hence,

\medskip

\begin{tabular}{ll}
$\tilde{\ep}^{\, 0}_\a$ & $=(\bar{s}_{\CLO(X_\a)}^{\widehat{X_\a}},\PPP(\breve{\imath}_{\widehat{X_\a}}))\circ\ep\ap_\a$\\
&$=(\bar{s}_{\CLO(X_\a)}^{\widehat{X_\a}},\PPP(\breve{\imath}_{\widehat{X_\a}}))\circ (\bar{s}_A^{X_\a},\widehat{\ep_B})$\\
&$=(\bar{s}_{\CLO(X_\a)}^{\widehat{X_\a}}\circ \bar{s}_A^{X_\a},\PPP(\breve{\imath}_{\widehat{X_\a}})\circ\widehat{\ep_B}),$\\
 \end{tabular}

  \medskip

\noindent where
$\widehat{\ep_B}:B\lra\PPP(\widehat{X_\a})$, $b\mapsto \{\widehat{\a_x}\st x\in\At(B), x\le b\}$.

Now we will define the contravariant functors $\FFF:\ZH\lra \MBool$ and  $\GGG:\MBool\lra\ZH$.

For every $X\in|\ZH|$, we put $$\FFF(X)\df i_X,$$ where $i_X:\co(X)\lra\PPP(X)$ is the inclusion map. Set $\a\df i_X$. Obviously, $\a$ is a z-map. Further, for every $x\in X=\At(\PPP(X))$, $\a_x:\co(X)\lra\2$ and $\a_x(U)=1\Leftrightarrow x\in\a(U)$, for every $U\in\co(X)$. Since $\a(U)=U$, we obtain that $\a_x=\hat{x}$ and thus $X_\a=\hat{X}$. For every $U\in\co(X)$, we have that $s_{\co(X)}^{\hat{X}}(U)=\{\hat{x}\st x\in X,\ \hat{x}(U)=1\}=\{\hat{x}\st x\in U\}=\hat{U}=\hat{h}_X(U)$.
Thus $s_{\co(X)}^{\hat{X}}(\co(X))=\hat{h}_X(\co(X))=\co(\hat{X})$ because $\hat{h}_X:X\lra \hat{X}$ is a homeomorphism (as it is shown in Example \ref{dzalgez}).
Hence, $i_X$ is an mz-map.

For $f\in\ZH(X,Y)$, we set $$\FFF(f)\df (\co(f),\PPP(f)).$$
Obviously, $\FFF(f)$ is a $\MBool$-morphism.

For $(\a:A\lra B)\in|\MBool|$, we put $$\GGG(\a)\df X_\a.$$
Clearly, the set $X_\a$ endowed with the subspace topology from the space $\SSS(A)$ is a $\ZH$-object.

For $(\p,\s)\in\MBool(\a,\a\ap)$, we set $$\GGG(\p,\s)\df f_\s.$$
The fact that $f_\s$ is a continuous map was proved in Theorem \ref{zeq} after the definition of $G\ap$ on the morphisms.

We define a natural isomorphism $\tau:\FFF_0\lra\FFF$ by $$\tau_X\df (id_{\CLO(X)},\bi_X^P)$$ for every $X\in|\ZH|$, where $\bi_X:\hat{X}\lra X$, $\hat{x}\mapsto x$, and $\bi_X^P:\PPP(\hat{X})\lra \PPP(X)$, $\hat{M}\mapsto\{\bi_X(\hat{m})\st \hat{m}\in\hat{M}\}$, (i.e., $\bi_X^P(\hat{M})=M$, for every $M\sbe X$). Indeed, it is obvious that for every $X\in|\ZH|$, $\tau_X:\FFF_0(X)\lra\FFF(X)$ is a $\MBool$-isomorphism and that, for every $f\in\ZH(X,X\ap)$, the diagram
\begin{center}
$\xymatrix{\FFF_0(X\ap)\ar[rr]^{\FFF_0(f)}\ar[d]_{\tau_{X\ap}} && \FFF_0(X)\ar[d]^{\tau_X}\\
\FFF(X\ap)\ar[rr]^{\FFF(f)} && \FFF(X)
}$
\end{center}
is commutative.

Now we define a natural isomorphism $\tau\ap:\GGG_0\lra\GGG$ by $$\tau\ap_\a\df \bi_{X_\a}$$ for every $\a\in|\MBool|$. Indeed, for every $X\in|\ZH|$, the map $\bi_X:\hat{X}\lra X$ is a homeomorphism since $\bi_X=\hat{h}_X\inv$ and the map $\hat{h}_X:X\lra \hat{X}$ is a homeomorphism;
hence, $\tau\ap_\a:\GGG_0(\a)\lra\GGG(\a)$ is a $\ZH$-isomorphism. Also, it is clear that, for every $(\p,\s)\in\MBool(\a,\a\ap)$,
the diagram
\begin{center}
$\xymatrix{\GGG_0(\a\ap)\ar[rr]^{\GGG_0(\p,\s)}\ar[d]_{\tau\ap_{\a\ap}} && \GGG_0(\a)\ar[d]^{\tau\ap_\a}\\
\GGG(\a\ap)\ar[rr]^{\GGG(\p,\s)} && \GGG(\a)
}$
\end{center}
is commutative.

Hence, we obtain that $\tau\ap\ast\tau:\GGG_0\circ\FFF_0\lra\GGG\circ\FFF$, where $(\tau\ap\ast\tau)_X=\tau\ap_{\FFF(X)}\circ\GGG_0(\tau_X\inv)$ for every $X\in|\ZH|$, is a natural isomorphism (see, e.g., \cite[Exercise 6A]{AHS}) and thus
$$\tilde{\eta}=(\tau\ap\ast\tau)\circ\tilde{\eta}^{\, 0}:\Id_{\ZH}\lra\GGG\circ\FFF$$ is a natural isomorphism. Analogously, $\tau\ast\tau\ap:\FFF_0\circ\GGG_0\lra\FFF\circ\GGG$, where $(\tau\ast\tau\ap)_\a=\tau_{\GGG(\a)}\circ\FFF_0((\tau\ap_\a)\inv)$ for every $\a\in|\MBool|$, is a natural isomorphism  and thus
$$\tilde{\ep}=(\tau\ast\tau\ap)\circ\tilde{\ep}^{\, 0}:\Id_{\MBool}\lra\FFF\circ\GGG$$ is a natural isomorphism.
Therefore, $\FFF$ and $\GGG$ are dual equivalences. It is now easy to obtain that, for every $X\in|\ZH|$ and every $x\in X$,
$$\tilde{\eta}_X(x)=\{\hat{x}\}$$ and, for every $(\a:A\lra B)\in|\MBool|$, $$\tilde{\ep}_\a=(\bar{s}_A^{X_\a},\ep_B^\a),$$
where $\ep_B^\a:B\lra\PPP(X_\a)$, $b\mapsto\{\a_x\st x\in\At(B), x\le b\}$, for every $b\in B$.
\end{proof}

\section{The restrictions of $F$ and $\FFF$ to the ca\-tegory $\ZHC$ imply the Stone Duality}

In this section we will derive the Stone Duality Theorem from Theorems \ref{zduality} and \ref{nzduality}. Of course, this is almost a formal act because in the proofs of these theorems  we have already utilized many facts which are parts of the proof of the Stone Duality Theorem. But doing this, we will show that our duality functors can be regarded as extensions of the Stone duality functors.

 Let us denote by $\CDZA$ the full subcategory
of the category $\DZA$ having as objects all compact dz-algebras (see Example \ref{dzalge} for this notion).

\begin{pro}\label{cdza}
The categories $\Bool$ and $\CDZA$ are isomorphic.
\end{pro}

\begin{proof}
Define a functor $E:\Bool\lra\CDZA$ by setting $E(A)\df (A,X_A)$,
for every $A\in|\Bool|$ (see Example \ref{dzalge} (or
\ref{nistone}) for the notation), and $E(\p)\df (\p,\SSS (\p))$, for
every $\Bool$-morphism $\p$. Then, by Example \ref{dzalge},
$E(A)\in|\CDZA|$ for every $A\in|\Bool|$. If $\p\in\Bool(A,A\ap)$,
then $(\SSS (\p))(x\ap)=x\ap\circ\p$  for every $x\ap\in X_{A\ap}$
(see \ref{nistone}). Hence $E(\p)\in\CDZA(E(A),E(A\ap))$.

Define also a functor $E\inv:\CDZA\lra\Bool$ by setting
$E\inv(A,X_A)\df A$, for every $(A,X_A)\in|\CDZA|$, and
$E\inv(\p,f)\df \p$, for every $\CDZA$-morphism $(\p,f)$.
It is easy to see that $E\circ E\inv=\Id_{\CDZA}$ and $E\inv\circ
E=\Id_{\Bool}$. (Indeed, it is enough to notice that if $(\p,f)$ is
a $\CDZA$-morphism then, by the definition of $\SSS (\p)$ (see
\ref{nistone}), we have that $f=\SSS (\p)$.) Thus $E$ and $E\inv$ are
isomorphisms.
\end{proof}

\begin{pro}\label{extstd}
Let $E^s:\ZHC\hookrightarrow\ZH$ and $E^a:\CDZA\hookrightarrow\DZA$ be the inclusion
functors.
Then $$F(E^s(|\ZHC|))\sbe |\CDZA|\ \mbox{ and }\ G(E^a(|\CDZA|))\sbe |\ZHC|.$$
Thus the restrictions $F_s:\ZHC\lra \CDZA$ and  $G_s:\CDZA\lra\ZHC$ of $F$ and $G$, respectively, are dual equivalences.
Also, $\TTT=E\inv\circ F_s$ and $\SSS= G_s\circ E$. Thus, $\TTT$ and $\SSS$ are dual equivalences.
Finally,  $F\circ E^s=E^a\circ E\circ \TTT $ and $E^s\circ\SSS=G\circ E^a\circ E$.
Therefore, the dual equivalences $F$ and $G$ are  extensions of
the dual equivalences $\TTT $ and $\SSS$, respectively.
(See   Theorem \ref{zduality}, Proposition \ref{cdza} and \ref{nistone} for the notation.)
\end{pro}

\begin{proof}
Let $X\in|\ZHC|$. Then $F(E^s(X))=F(X)=(\co(X),\hat{X})$. Since $X$ is compact, we have, as it is well-known, that $\hat{X}=\Bool(\co(X),\2)$. (Indeed, for every $\p\in \Bool(\co(X),\2)$, $\bigcap\{U\in\co(X)\st\p(U)=1\}$ is a singleton.) Thus $F(E^s(X))\in |\CDZA|$. Further, for every $(A,X_A)\in |\CDZA|$, $G(E^a(A,X_A))=G(A,X_A)=X_A=\SSS(A)$ and, as it is proved by M. Stone \cite{ST}, $G(E^a(A,X_A))\in  |\ZHC|.$ Thus, Theorem \ref{zduality} implies that $F_s$ and $G_s$ are dual equivalences. The equalities $\TTT=E\inv\circ F_s$ and $\SSS= G_s\circ E$ are obvious and hence, $\SSS\circ\TTT= G_s\circ E\circ E\inv\circ F_s=G_s\circ F_s\cong \Id_{\ZHC}$; analogously, $\TTT\circ\SSS\cong \Id_{\Bool}$. Therefore, $\TTT$ and $\SSS$ are dual equivalences.
Finally, we have that $E^a\circ E\circ \TTT= E^a\circ E\circ E\inv\circ F_s=E^a\circ F_s=F\circ E^s$ and $E^s\circ\SSS=E^s\circ G_s\circ E=
G\circ E^a\circ E$.
\end{proof}

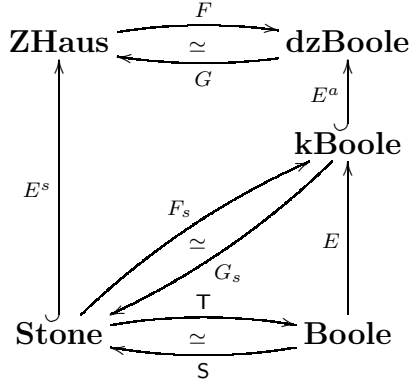
\begin{center}
$\xymatrix{\ZH\ar@/^0.6pc/[rr]^{F} & {\scriptstyle \simeq} & \DZA\ar@/^0.6pc/[ll]^{G}\\
&&\ar@/^0.6pc/[lldd]^{G_s}\CDZA\ar@{^(->}[u]^{E^a}\\
&{\scriptstyle \simeq}&\\
\ZHC\ar@{^(->}[uuu]^{E^s}\ar@/^0.6pc/[rr]^{\TTT}\ar@/^0.6pc/[rruu]^{F_s}&{\scriptstyle \simeq}&\Bool\ar[uu]^E\ar@/^0.6pc/[ll]^{\SSS}
}$
\end{center}

We are  now going to derive the Stone Duality Theorem from Theorem \ref{nzduality}.

Let  $\KMBool$ be the full subcategory of the category $\MBool$ having as objects all compact mz-maps (see Example \ref{mzmapek} for this notion).

\begin{pro}\label{mzbool}
 The categories\/ $\Bool$ and\/ $\KMBool$ are isomorphic.
\end{pro}

\begin{proof}
Let us define a functor $K:\Bool\lra\KMBool$ by setting $K(A)\df s_A^{X_A}$ for every $A\in|\Bool|$ (here $X_A=\Bool(A,\2)$),  and $K(\p)\df (\p,\PPP(\SSS(\p)))$, for every $\p\in\Bool(A,A\ap)$. Then Example \ref{mzmapek} shows that $K$ is well-defined on the objects. For proving that $K(\p)$ is a $\KMBool$-morphism, we have to verify the equality $s_{A\ap}^{X_{A\ap}}\circ\p=\PPP(\SSS(\p))\circ s_A^{X_A}$. Let $a\in A$. Then $(\PPP(\SSS(\p))\circ s_A^{X_A})(a)=(\SSS(\p))\inv(s_A^{X_A}(a))=\{x\ap\in X_{A\ap}\st (\SSS(\p))(x\ap)\in s_A^{X_A}(a)\}=\{x\ap\in X_{A\ap}\st x\ap(\p(a))=1\}=(s_{A\ap}^{X_{A\ap}}\circ\p)(a)$. Hence, $K$ is well-defined on morphisms as well. Obviously, $K$ is a functor. (Note that the use of the contravariant functors $\SSS$ and $\TTT$ can be easily avoided; we used them just for a simplification of the notation.)

Let us now define a functor $K\inv:\KMBool\lra \Bool$ by setting $K\inv(s_A^{X_A})\df A$ for every $A\in|\Bool|$, and $K\inv(\p,\s)\df \p$ for every $\KMBool$-morphism $(\p,\s)$. Then, obviously, $K\inv$ is a well-defined functor. It is clear that $K\inv\circ K=\Id_{\Bool}$ and $(K\circ K\inv)(s_A^{X_A})=s_A^{X_A}$ for every $A\in|\Bool|$. For every $\KMBool$-morphism $(\p,\s)$, we have $(K\circ K\inv)(\p,\s)=K(\p)=(\p,\PPP(\SSS(\p))$. Since $s_A^{X_A}\upharpoonright A=s_A:A\lra \co(X_A)$ is a Boolean isomorphism, the above calculation shows that $\s|\co(X_A)\equiv \PPP(\SSS(\p))|\co(X_A)$. Since every atom of $\PPP(X_A)$ (i.e., every element of $X_A$) is a meet in $\PPP(X_A)$ of some elements of $\co(X_A)$  and $\s$ is a complete homomorphisms, we see that $\s$ is uniquely  determined by its restriction on $\co(X_A)$. Therefore, $\s\equiv\PPP(\SSS(\p))$. Thus, $K\circ K\inv=\Id_{\KMBool}$. Hence, the categories\/ $\Bool$ and\/ $\KMBool$ are isomorphic.
\end{proof}

Now, using arguments similar to those used in the proof of Proposition \ref{extstd}, we obtain the following assertion:

\begin{pro}\label{extstdm}
Let  $E^m:\KMBool\hookrightarrow\MBool$ be the inclusion
functor.
Then $$\FFF(E^s(|\ZHC|))\sbe |\KMBool|\ \mbox{ and }\ \GGG(E^m(|\KMBool|))\sbe |\ZHC|.$$
Thus the restrictions\/ $\FFF_s:\ZHC\lra \KMBool$ and  $\GGG_s:\KMBool\lra\ZHC$ of\/ $\FFF$ and $\GGG$, respectively, are dual equivalences.
Also, $\TTT=K\inv\circ \FFF_s$ and\/ $\SSS= \GGG_s\circ K$. Thus, $\TTT$ and\/ $\SSS$ are dual equivalences.
Finally,  $\FFF\circ E^s=E^m\circ K\circ \TTT $ and\/ $E^s\circ\SSS=\GGG\circ E^m\circ K$.
Therefore, the dual equivalences $\FFF$ and\/ $\GGG$ are  extensions of
the dual equivalences\/ $\TTT $ and\/ $\SSS$, respectively.
(See   Theorem \ref{nzduality}, Propositions \ref{mzbool} and \ref{extstd},  and \ref{nistone} for the notation.)
\end{pro}

\begin{center}
$\xymatrix{\ZH\ar@/^0.6pc/[rr]^{\FFF} & {\scriptstyle \simeq} & \MBool\ar@/^0.6pc/[ll]^{\GGG}\\
&&\ar@/^0.6pc/[lldd]^{G_s}\KMBool\ar@{^(->}[u]^{E^m}\\
&{\scriptstyle \simeq}&\\
\ZHC\ar@{^(->}[uuu]^{E^s}\ar@/^0.6pc/[rr]^{\TTT}\ar@/^0.6pc/[rruu]^{F_s}&{\scriptstyle \simeq}&\ar@/^0.6pc/[ll]^{\SSS}\Bool\ar[uu]^K
}$
\end{center}

\section{The restrictions of $F$ and $\FFF$ to the ca\-tegory $\Di$ imply the Tarski Duality}

We are going to derive the Tarski Duality Theorem from Theorems \ref{zduality} and \ref{nzduality}. Unlike the previous section, this is not a formal act: it is true that in the proof of Theorem \ref{nzduality} we utilized some facts which are parts of the proof of the Tarski Duality Theorem, but the proof of Theorem \ref{zduality} is completely independent of the Tarski Duality Theorem.

 It is clear that the category $\Di$ of discrete spaces and continuous maps is a full subcategory of the category $\ZH$. Using the duality theorems proved in  Sections 3 and 4, we will find two categories dually equivalent to the category $\Di$. Since, obviously, the categories $\Di$ and $\Set$ are isomorphic, we will obtain in this way two
categories dually equivalent to the category $\Set$. Both of them will lead to one and the same  dual equivalence $\AAA:\Caba\lra\Set$ which will be slightly different from the Tarski dual equivalence $\Att:\Caba\lra\Set$ (and, maybe, will be new). From it we will easily obtain the Tarski Duality Theorem.

Let us denote by $\TDZA$ the full subcategory of the category $\DZA$ having as objects all T-algebras (see Example \ref{dzalget} for this notion), and
 let $\TMBool$ be the full subcategory of the category $\MBool$ having as objects all T-maps (see Example \ref{mzmapt} for this notion).

\begin{pro}\label{resttar}
 The categories\/ $\TDZA$ and\/ $\TMBool$ are dually equivalent to the category\/ $\Di$ (and, thus, to the category\/ $\Set$).
\end{pro}

\begin{proof}
Using the notation from the proofs of Theorems \ref{zduality} and \ref{nzduality}, it is enough to show that $F(|\Di|)\sbe |\TDZA|$, $G(|\TDZA|)\sbe|\Di|$,
$\FFF(|\Di|)\sbe |\TMBool|$ and $\GGG(|\TMBool|)\sbe|\Di|$.

We have that for every $X\in|\Di|$, $F(X)=(\co(X),\hat{X})=(\PPP(X),\hat{X})=(B,\check{X}_B)$, where $B\df\PPP(X)$, and, obviously, $(B,\check{X}_B)\in|\TDZA|$. Also, $\FFF(X)=i_X=id_{\, \PPP(X)}\in|\TMBool|$. Further, for every $(B,\check{X}_B)\in|\TDZA|$, $G(B,\check{X}_B)=\check{X}_B$, where $\check{X}_B$ is regarded as a subspace of $\SSS(B)$. Then, as it was shown in Example \ref{dzalget}, $\check{X}_B\in|\Di|$. Finally, for every $id_B\in|\TMBool|$, $\GGG(id_B)=X_{id_B}=\check{X}_B\in|\Di|$. Now  Theorems \ref{zduality} and \ref{nzduality} show that the restrictions $F_d:\Di\lra\TDZA$, $G_d:\TDZA\lra\Di$,
 $\FFF_d:\Di\lra\TMBool$, $\GGG_d:\TMBool\lra\Di$ of the contravariant functors $F$, $G$, $\FFF$ and $\GGG$, respectively, are all dual equivalences.
\end{proof}

\begin{cor}\label{cortar}
 For every\/ $\TDZA$-morphism $(\s,f)$ between any two $\TDZA$-objects   $(B,\check{X}_B)$ and $(B\ap,\check{X}_{B\ap})$, we have that $\s\in\Caba(B,B\ap)$.
\end{cor}

\begin{proof}
For every $f\in\Di(X,Y)$, we have that $F_d(f)=F(f)=(\co(f),\hat{f})=(\PPP(f),\hat{f})$. Since $\PPP(f)$ is a $\Caba$-morphism and $F_d$ is full, faithful and isomorphism-dense, our assertion follows.
\end{proof}

\medskip

We can prove this assertion directly, as well. Suppose that $\s$ is not a complete homomorphism. Then there exists a set $\{b_j\st j\in J\}\sbe B$ such that, with $b\df\bigvee_{j\in J} b_j$, $\s(b)\gneqq \bigvee_{j\in J}\s(b_j)$. Thus, there exists $y\in\At(B\ap)$ such that $y\le\s(b)$ but $y\nleq b\ap$, where $b\ap\df \bigvee_{j\in J}\s(b_j)$. Then $\check{y}(b\ap)=0$ and $\check{y}(\s(b))=1$. Since $\check{y}$ is a complete homomorphism (see \ref{hats}), we have that
$0=\check{y}(b\ap)=\check{y}(\bigvee_{j\in J}\s(b_j))=\bigvee_{j\in J}\check{y}(\s(b_j))=\bigvee_{j\in J}(\check{y}\circ\s)(b_j))$. Hence, $\bigvee_{j\in J}(\check{y}\circ\s)(b_j))\neq (\check{y}\circ\s)(\bigvee_{j\in J} b_j)$. Since $\check{y}\circ\s$  is a complete homomorphism  (because $\check{y}\circ\s=f(\check{y})\in\check{X}_B$), we obtain a contradiction. Therefore, $\s\in\Caba(B,B\ap)$.

\begin{nist}\label{newtar}
\rm
Using the above Corollary, we can define a functor $$H: \TDZA\lra\Caba$$ setting $H(B,\check{X}_B)\df B$ and $H(\s,f)\df \s$. Let us also define a functor
$$H\inv:\Caba\lra\TDZA$$ by $H\inv(B)\df (B,\check{X}_B)$ and, for any $\s\in\Caba(B,B\ap)$, $H\inv(\s)\df (\s,f^\s)$, where the function $$f^\s:\check{X}_{B\ap}\lra \check{X}_B$$ is defined by $$f^\s(\check{y})\df \check{y}\circ\s,$$ for every $\check{y}\in\check{X}_{B\ap}$. We need to show that $f^\s(\check{y})$ belongs to $\check{X}_B$. Indeed, setting $x\df \bigwedge\{a\in B\st y\le\s(a)\}$, we have that $x\in\At(B)$ and, using Lemma \ref{taroblm}, we obtain that   for every $b\in B$, $\check{x}(b)=1\Leftrightarrow x\le b\Leftrightarrow \bigwedge\{a\in B\st y\le\s(a)\}\le b\Leftrightarrow y\le\s(b)\Leftrightarrow\check{y}(\s(b))=1$. Thus $f^\s(\check{y})=\check{y}\circ\s=\check{x}\in \check{X}_B$. Hence, the functor $H\inv$ is well defined. One sees immediately that the compositions of the functors $H$ and $H\inv$ are equal to the corresponding identity functors. Therefore, $H$ and $H\inv$ are isomorphisms. Denoting by $I:\Di\lra\Set$ the obvious forgetful functor, we obtain that $I$ is an isomorphism and $H\circ F_d\circ I\inv=\PPP$. Now we set $$\AAA\df I\circ G_d\circ H\inv.$$ Using  Proposition \ref{resttar}, we obtain that $\PPP\circ\AAA=(H\circ F_d\circ I\inv)\circ(I\circ G_d\circ H\inv)=H\circ (F_d\circ G_d)\circ H\inv\cong H\circ \Id_{\TDZA}\circ H\inv=\Id_{\Caba}$ and, similarly, $\AAA\circ\PPP\cong \Id_{\Set}$. Thus, {\em the contravariant functors
$$\PPP:\Set\lra\Caba\ \mbox{ and } \ \AAA:\Caba\lra\Set$$ are dual equivalences.} Note that for every $B\in|\Caba|$, $$\AAA(B)=\check{X}_B,$$
where $\check{X}_B=\{\check{x}:B\lra\2\st x\in\At(B)\}$, $\check{x}(b)=1\Leftrightarrow x\le b$, and, for every $\s\in\Caba(B,B\ap)$, $$\AAA(\s)=f^\s$$
(see the definition of $f^\s$ here above).
It is easy to see that $\check{h}:\Att\lra\AAA$, where for every $B\in|\Caba|$, $\check{h}_B$ is the bijection defined in \ref{hats}, is a natural isomorphism. Thus, $\PPP\circ\Att\cong\PPP\circ\AAA\cong \Id_{\Caba}$ and, similarly, $\Att\circ\PPP\cong \Id_{\Set}$. Therefore, $\Att:\Caba\lra\Set$ and $\PPP:\Set\lra\Caba$ are dual equivalences, {\em obtaining in such a way a new proof of the Tarski Duality Theorem.}

Finally, defining a functor $H_1: \TMBool\lra\Caba$ by $H_1(id_B)\df B$ and $H_1(\s,\s)\df\s$, and a functor $H_1\inv:\Caba\lra\TMBool$ by $H\inv(B)\df id_B$ and  $H_1\inv(\s)\df (\s,\s)$ (note that Example \ref{dzalget} shows that $H_1\inv$ is well defined), we obtain that the compositions of the functors $H_1$ and $H_1\inv$ are equal to the corresponding identity functors. Therefore, $H_1$ and $H_1\inv$ are isomorphisms. Obviously, we get that  $H_1\circ \FFF_d\circ I\inv=\PPP$ and $\AAA= I\circ \GGG_d\circ H_1\inv.$ Hence, working with the contravariant functors $\FFF_d$ and $\GGG_d$, we come to the same dual equivalences $\PPP:\Set\lra\Caba$ and $\AAA:\Caba\lra\Set$.
\end{nist}

\section{Two duality theorems for the category $\EDT$ of extremally disconnected  spaces}

Now, using our duality theorems \ref{zduality} and \ref{nzduality}, we will obtain duality theorems for the category $\EDT$ of extremally disconnected Tychonoff spaces and continuous maps.

\begin{defi}\label{dzc}
\rm
A dz-algebra (resp., z-algebra) $(A,X)$ is said to be {\em complete dz-algebra}\/ (resp., {\em complete z-algebra}\/) if $A$ is a complete Boolean algebra.
Let us denote by $\DZCB$ the full subcategory of the category $\DZA$ having as objects all complete dz-algebras.
Let $\ZCB$ be the full subcategory of the category $\ZB$ having as objects of all complete z-algebras, and
let $\EDT$ be the category of extremally disconnected Tychonoff spaces and continuous maps.
\end{defi}

\begin{theorem}\label{zdualityed}
The categories $\EDT$ and $\ZCB$ are dually equivalent.
\end{theorem}

\begin{proof}
Since $\EDT$ is a subcategory of $\ZH$, we can regard the restriction $F_{ed}$ of the contravariant functor $F:\ZH\lra\DZA$ to $\EDT$. Analogously, we can regard the restriction $G_{ed}$ of the contravariant functor
$G:\DZA\lra\ZH$ to $\DZCB$. Recall that $F$ and $G$ were defined in the proof of Theorem \ref{zduality}. We will show that $F_{ed}(|\EDT|) \sbe|\DZCB|$ and $G_{ed}(|\DZCB|)\sbe|\EDT|$. Indeed, for every $X\in|\EDT|$, we have that $\CO(X)=\RC(X)$ and thus $F_{ed}(X)=(\co(X),\hat{X})=(\RC(X),\hat{X})$. Hence, $F(X)\in |\DZCB|$. If $(A,X)\in|\DZCB|$, then $G_{ed}(A,X)=X$. Since, by Fact \ref{zalgf}, $X$ is a dense subspace of the extremally disconnected space $\SSS(A)$, we obtain that $X$ is an extremally disconnected space (see, e.g., \cite[Exercise 6.2.G.(c)]{E}). Thus, $G_{ed}(A,X)\in|\EDT|$.
Now, Theorem \ref{zduality} implies that $$F_{ed}:\EDT\lra\DZCB \ \mbox{ and } \ G_{ed}:\DZCB\lra\EDT$$ are dual equivalences.
Finally, we will show that the categories $\DZCB$ and $\ZCB$ coincide. Indeed, if $(A,X)\in|\ZCB|$, then, using Lemma \ref{isombool}, we obtain that $s_A^X(A)=X\cap s_A(A)=X\cap\co(\SSS(A))=X\cap\RC(\SSS(A))=\RC(X)=\co(X)$. Therefore, $(A,X)$ is a dz-algebra. Thus, the categories $\EDT$ and\/ $\ZCB$ are dually equivalent.
\end{proof}

\begin{defi}\label{mzc}
\rm
An mz-map (resp., z-map) $\a:A\lra B$ is said to be {\em complete mz-map}\/ (resp., {\em complete z-map}\/) if $A$ is a complete Boolean algebra.
Let us denote by $\CMZM$ the full subcategory of the category $\MBool$ having as objects all complete mz-maps,
and by $\CZM$  the full subcategory of the category $\ZBool$ having as objects of all complete z-maps.
\end{defi}

\begin{theorem}\label{mzdualityed}
The categories $\EDT$ and $\CZM$ are dually equivalent.
\end{theorem}

\begin{proof}
 Let us denote by $\FFF_{ed}$  the restriction of the contravariant functor $\FFF:\ZH\lra\MBool$ to $\EDT$, and by $\GGG_{ed}$  the restriction of the contravariant functor
$\GGG:\MBool\lra\ZH$ to $\CMZM$. Recall that $\FFF$ and $\GGG$ were defined in the proof of Theorem \ref{nzduality}. We are going to show that $\FFF_{ed}(|\EDT|) \sbe|\CMZM|$ and $\GGG_{ed}(|\CMZM|)\sbe|\EDT|$. Indeed, for every $X\in|\EDT|$, we have that $\FFF_{ed}(X)=i_X$, where $i_X:\co(X)\hookrightarrow\PPP(X)$ is the inclusion map. Since $\CO(X)=\RC(X)$, we obtain that  $\FFF_{ed}(X)\in |\CMZM|$. Let now $(\a:A\lra B)\in|\CMZM|$. Then  $\GGG_{ed}(\a)=X_\a$.
 We will show that $X_\a$ is a dense subspace of $\SSS(A)$. Indeed, if $a\in A^+$ then $\a(a)\neq 0$ and, hence, there exists $x\in\At(B)$ such that $x\le \a(a)$; this, however, means that $\a_x(a)=1$, i.e., $\a_x\in s_A(a)\cap X_\a$. So,  $X_\a$ is a dense subspace of $\SSS(A)$.
Thus, $\GGG_{ed}(\a)\in |\EDT|$. Now, Theorem \ref{nzduality} implies that $$\FFF_{ed}:\EDT\lra\CMZM \ \mbox{ and } \ \GGG_{ed}:\CMZM\lra\EDT$$ are dual equivalences.
Finally, we will show that the categories $\CMZM$ and $\CZM$ coincide. Indeed, let $\a:A\lra B$ be a complete z-map. Then $A$ is a complete Boolean algebra and, hence, $\SSS(A)$ is extremally disconnected. As we have already seen, $X_\a$ is a dense subspace of  $\SSS(A)$, and thus $X_\a$ is also extremally disconnected. Now, using Lemma \ref{isombool}, we obtain that $s_A^{X_\a}(A)=X_\a\cap s_A(A)=X_\a\cap\co(\SSS(A))=X_\a\cap\RC(\SSS(A))=\RC(X_\a)=\co(X_\a)$. Therefore, $\a$ is an mz-map.
This shows that $\CMZM\equiv\CZM$. Hence, the categories $\EDT$ and $\CZM$ are dually equivalent.
\end{proof}

\section{Two duality theorems for the category of zero-di\-men\-sional Hausdorff compactifications
of zero-di\-men\-sional spaces}

Recall first the following assertion from \cite{BMO}:

\begin{pro}\label{bpro}{\rm\cite{BMO}}
There is a category $\Comp$ whose objects are Hausdorff
compactifications $c:X\lra Y$ and whose morphisms between any two
$\Comp$ -objects $c:X\lra Y$ and $c\ap:X\ap\lra Y\ap$ are all
pairs $(f,g)$, where $f:X\lra X\ap$ and $g:Y\lra Y\ap$ are
continuous maps such that $g\circ c=c\ap\circ f$. The composition
of two morphisms $(f_1,g_1)$ and $(f_2,g_2)$ is defined to be
$(f_2\circ f_1, g_2\circ g_1)$. The identity map of a
$\Comp$-object  $c:X\lra Y$ is defined to be $id_c\df
(id_X,id_Y)$.
\end{pro}

\begin{defi}\label{zcompd}
\rm We will denote by $\ZComp$ the full subcategory of the
category $\Comp$ whose objects are all Hausdorff compactifications
$c:X\lra Y$ for which $Y$ is a zero-dimensional space.
\end{defi}

\begin{rem}\label{zcompe}
\rm Note that Example 3.2 from \cite{BMO} shows that there exist
$\ZComp$-objects $c:X\lra Y$ and $c\ap:X\lra Y\ap$ which are
isomorphic in $\ZComp$ but not equivalent as compactifications. On
the other hand, as it is shown in \cite{BMO}, any two equivalent
compactifications of a space $X$ are isomorphic in $\Comp$.
\end{rem}


\begin{pro}\label{zcompb}
Let $c:X\lra Y$ be a $\ZComp$-object. If $c$ is isomorphic to the
Banaschewski compactification $\b_0:X\lra\b_0X$ in $\ZComp$, then
$c$ is equivalent to $\b_0$.
\end{pro}

\begin{proof}
The proof is analogous to that of Theorem 3.3 from \cite{BMO}. The
only difference is that the Banaschewski Theorem \ref{zdextcb} has
to be used.
\end{proof}

\begin{theorem}\label{zdualityc}
The categories $\ZComp$ and\/ $\ZA$ are dually equivalent.
\end{theorem}

\begin{proof}
We start by defining a contravariant functor $$\Phi:\ZComp\lra\ZA.$$

For every $(c:X\lra Y)\in|\ZComp|$, set $A_c\df c\inv(\co(Y))$,
$\hat{X}_c\df \hat{X}_{A_c}$
(see  \ref{hats} for the notation),
and $$\Phi(c)\df (A_c,\hat{X}_c).$$ Then, by Example \ref{dzalgez},
$\Phi(c)\in|\ZA|$.

Let now $c:X\lra Y$ and $c\ap:X\ap\lra Y\ap$ be $\ZComp$-objects
and $(f,g)$ be a $\ZComp$-morphism between $c$ and $c\ap$. Set
 $$\Phi(f,g)\df (\pi_f,\hat{f}_{cc\ap}),$$ where
$\pi_f:A_{c\ap}\lra A_c$ is defined by $\pi_f(U)\df f\inv(U)$ for
every $U\in A_{c\ap}$, and $$\hat{f}_{cc\ap}:\hat{X}_c\lra \widehat{X\ap}_{c\ap}$$ is
defined by $\hat{f}_{cc\ap}(\hat{x})\df \widehat{f(x)}$ for every $x\in X$.
Arguing as in the
proof of Theorem \ref{zduality}, we obtain that
$\Phi(f,g)\in\ZA(\Phi(c\ap),\Phi(c))$. Now it is easy to see that
$\Phi$ is a contravariant functor.

We define $\Psi:\ZA\lra\ZComp$ as follows: for every $(A,X)\in|\ZA|$, set
$$\Psi(A,X)\df c_{(A,X)},$$
where, regarding $X$ as a subspace of $\SSS(A)$, $c_{(A,X)}:X\hookrightarrow\SSS(A)$ is the embedding of $X$ in $\SSS(A)$;
for every $(\p,f)\in\ZA((A,X),(A\ap,X\ap))$, we put
$$\Psi(\p,f)\df (f,\SSS(\p)).$$
By Fact \ref{zalgf}, $c_{(A,X)}$ is a dense embedding and thus $\Psi(A,X)$ is a $\ZComp$-object. Since for every $x\ap\in X\ap$, $\SSS(\p)(x\ap)=x\ap\circ\p=f(x\ap)$, we obtain that $\Psi(\p,f)$ is a $\ZComp$-morphism. Hence, $\Psi$ is well-defined. Obviously, it is a contravariant functor.

Let $(A,X)\in|\ZA|$. Then $\Phi(\Psi(A,X))=(A_{c_{(A,X)}}, \hat{X}_{c_{(A,X)}})$, $A_{c_{(A,X)}}=X\cap s_A(A)=s_A^X(A)$ and $\hat{X}_{c_{(A,X)}}=\{\hat{x}:s_A^X(A)\lra\2\st x\in X\}$. Working like in the proof of Theorem \ref{zduality},
we define a map ${\breve{\imath}}_X^{\, c_{(A,X)}}:\hat{X}_{c_{(A,X)}}\lra X$ by ${\breve{\imath}}_X^{\, c_{(A,X)}}(\hat{x})\df x$, for every $x\in X$, and set  $s''_{(A,X)}\df (\bar{s}_A^X,{\breve{\imath}}_X^{\, c_{(A,X)}})$. Then, like in Theorem \ref{zduality}, we show that $s''_{(A,X)}:(A,X)\lra (\Phi\circ \Psi)(A,X)$ is a $\ZA$-isomorphism and, moreover, $$s'':\Id_{\ZA}\lra \Phi\circ \Psi, \ \ (A,X)\mapsto s''_{(A,X)},$$ is a natural isomorphism.

Let now $(c:X\lra Y)\in|\ZComp|$. Then $(\Psi\circ \Phi)(c)=c_{(A_c,\hat{X}_c)}$ and $c_{(A_c,\hat{X}_c)}:\hat{X}_c\hookrightarrow\SSS(A_c)$.
Obviously, the map $\rho_c: A_c\lra \co(Y),\ \ c\inv(U)\mapsto U,$ is a Boolean isomorphism. Hence, the map $\SSS(\rho_c): \SSS(\TTT(Y))\lra\SSS(A_c)$ is a homeomorphism. By Example \ref{dzalgez}, the map $\hat{h}_{X,A_c}:X\lra \hat{X}_c$ is a homeomorphism.
Now it is easy to show that the map $\vk_c\df (\hat{h}_{X,A_c},\SSS(\rho_c)\circ t_Y):c\lra c_{(A_c,\hat{X}_c)}$ is a $\ZComp$-isomorphism (see \ref{nistone} for the notation $t_Y$). Finally, it is not difficult to prove that $$\vk:\Id_{\ZComp}\lra \Psi\circ \Phi, \ \ c \mapsto \vk_c,$$ is a natural isomorphism.
Therefore, the categories $\ZComp$ and\/ $\ZA$ are dually equivalent.
\end{proof}

We will denote by $\EDComp$ the full subcategory of the category $\ZComp$ having as objects all compactifications $c:X\lra Y$, for which $Y\in|\EDT|$.

\begin{cor}\label{coredtc}
The categories $\EDComp$ and $\ZCB$ are dually equivalent.
\end{cor}

\begin{proof}
Having in mind Theorem \ref{zdualityc},   it is enough to show that $\Phi(|\EDComp|)\sbe |\ZCB|$ and   $\Psi(|\ZCB|)\sbe |\EDComp|$.
Let  $(c:X\lra Y)\in|\EDComp|$. Then, using Lemma \ref{isombool}, we obtain (in the notation from the proof of Theorem \ref{zdualityc}) that $A_c=c\inv(\co(Y))=c\inv(\RC(Y))=\RC(X)$. Thus, $\Phi(c)\in|\ZCB|$. Let now $(A,X)\in|\ZCB|$. Then $A$ is a complete Boolean algebra and, hence, $\SSS(A)\in|\EDT|$. This shows that $\Psi(A,X)\in|\EDComp|$. So, the proof is completed.
\end{proof}

\begin{cor}\label{coredtcc}
The categories $\EDComp$ and $\EDT$ are equivalent.
\end{cor}

\begin{proof}
This follows immediately from Theorem \ref{zdualityed} and Corollary \ref{coredtc}.
\end{proof}

Note that  Corollary \ref{coredtcc} can be also proved  with the help of the fact that if $(c:X\lra Y)\in|\EDComp|$ then $X$ is extremally disconnected and $c$ is equivalent (as a compactification of $X$) to the Stone-\v{C}ech compactification $\b:X\lra\b X$ of $X$ (see \cite{GJ} or \cite{E}).

Now we will  show, using the
Tarski duality,  that the category $\ZBool$ is dually
equivalent to the category $\ZComp$. The category $\ZBool$ is
similar to the category $\DeVe$, constructed in \cite{BMO} as a
 category dually equivalent to the category $\Comp$ of Hausdorff
compactifications of Tychonoff spaces.

\begin{theorem}\label{nzdualityc}
The categories $\ZComp$ and\/ $\ZBool$ are dually equivalent.
\end{theorem}

\begin{proof}
We will utilize the notation introduced in the proof of Theorem \ref{zdualityc}.

We start by defining a contravariant functor $$\Phi\ap:\ZComp\lra\ZBool.$$

For every $(c:X\lra Y)\in|\ZComp|$, we set $$\Phi\ap(c)\df s_{A_c}^{\hat{X}_c}.$$  Then it is easy to see that
$\Phi\ap(c)\in|\ZBool|$.

For every $(f,g)\in\ZComp(c,c\ap)$, we set $$\Phi\ap(f,g)\df
(\pi_f,\PPP(\hat{f}_{cc\ap})).$$
It is not difficult to  obtain that
$\Phi\ap(f,g)\in\ZBool(\Phi\ap(A\ap,X\ap),\Phi\ap(A,X)).$ Now it is easy
to see that $\Phi\ap$ is a contravariant functor.

Our next aim is to define a contravariant functor
$$\Psi\ap:\ZBool\lra\ZComp.$$
Let $(\a:A\lra B)\in|\ZBool|$. We put
$$\Psi\ap(\a)\df c_\a,\ \mbox{ where }\ c_\a:X_\a\hookrightarrow\SSS(A)$$
(see \ref{hats} for the notation $X_\a$). Obviously, $\Psi\ap(c)\in|\ZComp|$.

Let now $(\p,\s)\in\ZBool(\a,\a\ap)$. Then it is easy to show that $\SSS(\p)(X_{\a\ap})\sbe X_\a$.
Let $S_\p:X_{\a\ap}\lra X_\a$ be the restriction of $\SSS(\p)$. We put
$$\Psi\ap(\p,\s)\df (S_\p,\SSS(\p)).$$
Then it is not difficult to prove that $\Psi\ap(\p,\s)\in\ZComp(\Psi\ap(\a\ap),\Psi\ap(\a))$ and that
$\Psi\ap$ is a contravariant functor.

Let $(\a:A\lra B)\in|\ZBool|$. Then $\Phi\ap(\Psi\ap(\a))=s_{A_{c_\a}}^{\hat{X}_{c_\a}}$ and $s_{A_{c_\a}}^{\hat{X}_{c_\a}}:A_{c_\a}\lra\PPP(\hat{X}_{c_\a})$.
We have that $A_{c_\a}=c_\a\inv(\co(\SSS(A)))=X_\a\cap\TTT(\SSS(A))=s_A^{X_\a}(A)$. Thus $\bar{s}_A^{X_\a}:A\lra A_{c_\a}$ is a Boolean isomorphism.
Since $\a$ is a z-map,  the map $h_\a:\At(B)\lra X_\a$, $x\mapsto \a_x$, is a bijection (see \ref{hats}).
Also, the map $\hat{h}_{X_\a,A_{c_\a}}: X_\a\lra\hat{X}_{c_\a}$, $\a_x\mapsto\widehat{\a_x}$, for all $x\in\At(B)$, where $\widehat{\a_x}:A_{c_\a}\lra\2$, is a bijection (see \ref{hats}).
 Setting $k_\a\df \hat{h}_{X_\a,A_{c_\a}}\circ h_\a$ and $k_\a^P:\PPP(\At(B))\lra\PPP(\hat{X}_{c_\a})$, $M\mapsto\{k_\a(m)\st m\in M\}$, we obtain that $k_\a^P$ is a bijection.
 Then the map $\ep_B^{c_\a}\df k_\a^P\circ\ep_B$ is a bijection (see \ref{tar} for the notation $\ep_B$) and
$\ep_B^{c_\a}:B\lra\PPP(\hat{X}_{c_\a})$,  $b\mapsto \{\widehat{\a_x}\st x\in\At(B), x\le b\}.$ Now we put $\upsilon_\a\df (s_A^{X_\a},\ep_B^{c_\a})$.  It is easy to see that $\upsilon_\a:\a\lra \Phi\ap(\Psi\ap(\a))$ is a $\ZBool$-isomorphism. One routinely verifies that
$$\upsilon:\Id_{\ZBool}\lra \Phi\ap\circ\Psi\ap,\ \ \  \a\mapsto\upsilon_\a,$$
is a natural isomorphism.

Let $(c:X\lra Y)\in|\ZComp|$. Then $\Psi\ap(\Phi\ap(c))=c_\a$, where $\a\df s_{A_c}^{\hat{X}_c}$. Thus $c_\a: X_\a\hookrightarrow \SSS(A_c)$, where $A_c=c\inv(\co(Y))$ and $X_\a=\{\a_{\hat{x}}\st \hat{x}\in \hat{X}_c\}$. We have that for every $U\in A_c$, $\a_{\hat{x}}(U)=1\Leftrightarrow \hat{x}\le \a(U)\Leftrightarrow \hat{x}\in s_{A_c}^{\hat{X}_c}(U)\Leftrightarrow \hat{x}(U)=1$. Thus, $\a_{\hat{x}}\equiv \hat{x}$ for every $x\in X$. Hence, $X_\a=\hat{X}_c$, i.e., $c_\a: \hat{X}_c\hookrightarrow \SSS(A_c)$. As we noted in the proof of Theorem \ref{zdualityc},
 the maps $\SSS(\rho_c): \SSS(\TTT(Y))\lra\SSS(A_c)$ (where $\rho_c: A_c\lra \co(Y),\ \ c\inv(U)\mapsto U$)
 and  $\hat{h}_{X,A_c}:X\lra \hat{X}_c, \ \ x\mapsto \hat{x},$ are homeomorphisms. Now it is easy to see that the map $\xi_c\df (\hat{h}_{X,A_c},\SSS(\rho_c)\circ t_Y):c\lra \Psi\ap(\Phi\ap(c))$ is a $\ZComp$-isomorphism (see \ref{nistone} for the notation $t_Y$). Finally, a routine verification shows that
$$\xi:\Id_{\ZComp}\lra \Psi\ap\circ\Phi\ap,\ \ \  c\mapsto\xi_c,$$
is a natural isomorphism.

All this proves that the categories $\ZComp$ and\/ $\ZBool$ are dually equivalent.
\end{proof}

\begin{nist}\label{dwthcor}
\rm
We are now going to derive the Dwinger Theorem \ref{dwingerth} from our Theorems \ref{zdualityc} and \ref{nzdualityc}.
In what follows, we will use the notation from their proofs.

Let us fix a space $X\in|\ZH|$. Then, obviously, the map $\l:\BB\AA(X)\lra |\ZA|, \ \  A\mapsto (A,\hat{X}_A),$ is an injection. (Note that, by Example \ref{dzalgez}, $\l$ is a well-defined function.) Thus, the map $\l_0\df\l\upharpoonright\BB\AA(X):\BB\AA(X)\lra \l(\BB\AA(X))$ is a bijection. We have that $\Psi(\l(A))=c_{(A,\hat{X}_A)}$, where
$c_{(A,\hat{X}_A)}:\hat{X}_A\hookrightarrow\SSS(A)$ is the embedding of $\hat{X}_A$ in $\SSS(A)$. We set $c_A\df c_{(A,\hat{X}_A)}\circ\hat{h}_{X,A}$ and $\DE(A)\df [c_A]$. Then $$c_A:X\lra\SSS(A),\ \ x\mapsto\hat{x},\ \ \mbox{ and }\ \ \DE:\BB\AA(X)\lra\KK_0(X).$$

For every $(c:X\lra Y)\in\KK_0(X)$, we set $\DE\ap([c])\df \l_0\inv(\Phi(c))$. Thus $$\DE\ap([c])=A_c=c\inv(\co(Y))\in\BB\AA(X)\ \ \mbox{ and }\ \ \DE\ap:\KK_0(X)\lra\BB\AA(X).$$ Note that the map $\DE\ap$ is well-defined. Indeed, if $c_1\in[c]$, where $c_1:X\lra Y_1$, then there exists a homeomorphism $f:Y\lra Y_1$ such that $c_1=f\circ c$. Hence $c_1\inv(\co(Y_1))=c\inv(f\inv(\co(Y_1)))=c\inv(\co(Y))$.

Now, for every $A\in\BB\AA(X)$, $$\DE\ap(\DE(A))=A.$$ Indeed, we have that $\DE\ap(\DE(A))=A_{c_A}=c_A\inv(\TTT(\SSS(A)))=\hs\inv_{X,A}(\bar{s}_A^{\hat{X}_A}(A))$ and, for every $U\in A$, $\hs\inv_{X,A}(\bar{s}_A^{\hat{X}_A}(U))
=U$ (see the proof of Example \ref{dzalgez}).

Further, for every $(c:X\lra Y)\in\KK_0(X)$, $\DE(\DE\ap([c]))=\DE(A_c)=[c_{A_c}]$, where $c_{A_c}:X\lra\SSS(A)$.
At the end of the proof of Theorem \ref{nzdualityc} we have shown that the map $(\hat{h}_{X,A_c},\SSS(\rho_c)\circ t_Y):c\lra c_{(A_c,\hat{X}_c)}$ is a $\ZComp$-isomorphism. Using the definition of the map $c_{A_c}$, we obtain that the map $(\hat{h}_{X,A_c}\inv,id_{\SSS(A_c)}): c_{(A_c,\hat{X}_c)}\lra c_{A_c}$ is also a $\ZComp$-isomorphism. Thus the  diagram

\begin{center}
$\xymatrix{X\ar[rr]_{\hat{h}_{X,A_c}}\ar@/^1.0pc/[rrrr]^{\rm id_{X}}\ar[dd]^{c} && \hat{X}_c\ar[rr]_{\hat{h}_{X,A_c}\inv}\ar[dd]_{c_{(A_c,\hat{X}_c)}} && X\ar[dd]_{c_{A_c}}\\
&&&&\\
            Y\ar[rr]_{\SSS(\rho_c)\circ t_Y} && \SSS(A_c)\ar[rr]_{id_{\SSS(A_c)}} && \SSS(A_c)}$
            \end{center}
is commutative. It shows that the compactifications $c$ and $c_{A_c}$ of $X$ are equivalent (since $c_{A_c}=(S(\rho_c)\circ t_Y)\circ c$). Thus, $$\DE(\DE\ap([c]))=[c].$$ Therefore, $\DE$ and $\DE\ap$ are bijections.

Let now $c_1:X\lra Y_1$ and $c_2:X\lra Y_2$ be compactifications of $X$, and $c_1\le c_2$. Then there exists a continuous map $g:Y_2\lra Y_1$ such that $c_1=g\circ c_2$. Thus, $(id_X,g)\in\ZComp(c_2,c_1)$. Then
 $A_{c_1}=c_1\inv(\co(Y_1))=c_2\inv(g\inv(\co(Y_1)))\sbe c_2\inv(\co(Y_2))=A_{c_2}$. Therefore, $$\DE\ap([c_1])\le\DE\ap([c_2]).$$

Let now $A,A\ap\in\BB\AA(X)$ and $A$ be a subalgebra of $A\ap$; denote by $i:A\lra A\ap$ the inclusion monomorphism. For every $x\in X$,  set $f(\hat{x}_{A\ap})=\hat{x}_{A}$ (see \ref{hats} for the notation). Then $f:\hat{X}_{A\ap}\lra \hat{X}_{A}$, $f(\hat{x}_{A\ap})=\hat{x}_{A\ap}\circ i$ and thus $(i,f)\in\ZA(\l(A),\l(A\ap))$. Therefore $\Psi(i,f):\Psi(\l(A\ap))\lra\Psi(\l(A))$ is a $\ZComp$-morphism. We have that $\Psi(i,f)=(f,\SSS(i))$. Hence, the diagram

\begin{center}
$\xymatrix{X\ar[rr]_{\hat{h}_{X,A\ap}}\ar@/^1.0pc/[rrrr]^{\rm c_{A\ap}}\ar[dd]^{id_X} && \hat{X}_{A\ap}\ar@{^(->}[rr]\ar[dd]_{f} && \SSS(A\ap)\ar[dd]_{\SSS(i)}\\
&&&&\\
            X\ar[rr]^{\hat{h}_{X,A}}\ar@/^-1.0pc/[rrrr]_{\rm c_A} && \hat{X}_A\ar@{^(->}[rr] && \SSS(A)}$
            \end{center}
is commutative. It shows that $$\DE(A)\le\DE(A\ap).$$ Therefore, $\DE$ and $\DE\ap$ are isomorphisms between the ordered sets $(\BB\AA(X),\sbe)$ and $(\KK_0(X),\le)$. Thus, the Dwinger Theorem is proved.

For deriving the Dwinger Theorem  from Theorem \ref{nzdualityc}, we use the same maps $\DE$ and $\DE\ap$ but find another expressions for them.  For every $(c:X\lra Y)\in\KK_0(X)$, we have that $\Phi\ap(c)=s_{A_c}^{\hat{X}_c}$ and thus $\DE\ap([c])=\dom(\Phi\ap(c))$. Also, for every $A\in\BB\AA(X)$, we have, by Example \ref{zmapt}, that the map $i_A:A\hookrightarrow\PPP(X)$ is a z-map. Set $\a\df i_A$. Then $\Psi\ap(\a)=c_{\a}$, where $c_\a:X_\a\hookrightarrow\SSS(A)$, and  $X_{\a}\equiv \hat{X}_A$. Thus $c_{\a}\equiv c_{(A,\hat{X}_A)}$ and  $\DE(A)=c_A=\Psi\ap(\a)\circ\hat{h}_{X,A}$. Then we prove exactly as above that  $\DE$ and $\DE\ap$ are bijections, and that $\DE\ap$ is monotone. Finally, let $A,A\ap\in\BB\AA(X)$ and $A\sbe A\ap$. Denote by $i:A\hookrightarrow A\ap$ the inclusion monomorphism and set $\a\df i_A$, $\a\ap\df i_{A\ap}$. Then $(i,id_{\,\PPP(X)})\in\ZBool(\a,\a\ap)$ and thus $\Psi\ap(i,id_{\,\PPP(X)})\in\ZComp(c_{\a\ap},c_\a)$. We have that $\Psi\ap(i,id_{\,\PPP(X)})=(S_i,\SSS(i))$, where $S_i:\hat{X}_{A\ap}\lra \hat{X}_A$ is the restriction of $\SSS(i):\SSS(A\ap)\lra\SSS(A)$. Writing in the last diagram $S_i$ instead of $f$, we obtain a new commutative diagram which shows again that $\DE(A)\le\DE(A\ap).$ Thus, the second proof of the Dwinger Theorem is completed.
\end{nist}


\end{document}